\documentclass[a4paper, 11pt]{article}
\usepackage[english]{babel}
\usepackage{amsmath, amssymb, amsfonts, amsthm, bm}
\usepackage{mathrsfs}
\usepackage[mathcal]{euscript}
\usepackage{enumerate}
\usepackage{float}
\usepackage{color,graphicx,epstopdf}
\usepackage{dsfont}
\usepackage{hyperref}
\usepackage{cite} 
\usepackage{algorithm} 
\usepackage{algorithmic} 
\usepackage{tikz-cd}
\usepackage{mathtools}
\usepackage{subcaption}
\usepackage{graphicx}

\newcommand{\R}{\mathbb{R}}

\DeclareFontFamily{OT1}{pzc}{}
\DeclareFontShape{OT1}{pzc}{m}{it}{<-> s * [1.200] pzcmi7t}{}
\DeclareMathAlphabet{\mathpzc}{OT1}{pzc}{m}{it}

\newcommand*\mcapinn[2]{\vcenter{\hbox{$\mathsurround=0pt
			\ifx\displaystyle#1\textstyle\else#1\fi\bigcap$}}}

\newcommand*\mcup{\mathbin{\mathpalette\mcupinn\relax}}
\newcommand*\mcupinn[2]{\vcenter{\hbox{$\mathsurround=0pt
			\ifx\displaystyle#1\textstyle\else#1\fi\bigcup$}}}
		
\DeclareMathOperator{\SO}{SO}
\DeclareMathOperator{\so}{\mathfrak{so}}

\newtheorem{theorem}{Theorem}
\newtheorem{definition}{Definition}

\newtheorem{lemma}{Lemma}
\newtheorem{remark}{Remark}
\newtheorem{proposition}{Proposition}

\title{\bf Controllability and Accessibility on Graphs for\\ Bilinear Systems over Lie Groups}
\date{}
 	
 	\author{Xing Wang\thanks{Key Laboratory of Mathematics Mechanization, Institute of Systems Science, Academy of Mathematics and Systems Science, Chinese Academy of Sciences, Beijing 100190, China; School of Mathematical Sciences, University of Chinese Academy of Sciences, Beijing 100049, China. (wangxing17@amss.ac.cn)}, Bo Li\thanks{Key Laboratory of Mathematics Mechanization, Academy of Mathematics and Systems Science, Chinese Academy of Sciences, Beijing 100190, China. (libo@amss.ac.cn)}, Jr-Shin Li\thanks{Department of Electrical and Systems Engineering, Washington University, St. Louis, MO 63130, USA. (jsli@wustl.edu)},
 	Ian R. Petersen\thanks{Research School of Electrical, Energy and Materials Engineering, Australian National University, Canberra, ACT 0200, Australia. (ian.petersen@anu.edu.au)}, Guodong Shi\thanks{Australian Center for Field Robotics, School of Aerospace, Mechanical and Mechatronic Engineering, The University of Sydney, NSW 2006, Australia. (guodong.shi@sydney.edu.au)}}
\begin{document}
 	\maketitle
 	
 	\begin{abstract}
 		This paper presents  graph theoretic conditions  for the controllability and accessibility of bilinear systems over the special orthogonal group, the special linear group and the general linear group, respectively, in the presence of drift terms. Such bilinear systems naturally induce two interaction graphs: one graph from the drift, and another from the controlled  dynamics. As a result, the system controllability or accessibility becomes a property of the two graphs in view of the classical Lie algebra rank condition. We establish a systemic way of transforming the Lie bracket operations in the underlying Lie algebra, into specific  operations of removing or creating links over the drift and controlled interaction  graphs.  As a result, we  establish a series of graphical conditions for the controllability and accessibility of such bilinear systems, which rely only on the connectivity of the union of  the drift and controlled interaction graphs. We present examples to illustrate the validity of the established results, and show that the proposed conditions are in fact considerably tight.
 	\end{abstract}
 
\section{Introduction}
Bilinear  systems are a simple yet important class of  nonlinear  dynamical systems, where the system evolution is influenced by the product of, and therefore depends bilinearly on, the system state and control actions \cite{Elliott2009}. The study of bilinear systems was originated in the 1970s with a primary interest in investigating dynamical systems with geometrical state constraints \cite{Elliott1971,Brockett1973,Aoki1975}.  A series of seminal works were initially  developed on the controllability and accessibility theory of bilinear systems over groups \cite{H1972,J1975,li2019}; and  the discovery of   the so-called  Lie algebra rank condition for bilinear controllability, opened  a new era of nonlinear system theory where differential geometry insights can be built into feedback control principles.  In subsequent years  the literature has found much improved theoretical understandings, and applications  of bilinear systems in the fields of engineering, economic, and quantum systems \cite{Khapalov1996,altafini2002,albertini2003,dir08,alb02,qi2020} .

In the past decade, a new thriving line of research in control systems has involved the control of multi-agent or networked systems, where subsystems are interconnected according to a graph representing the system interaction structure \cite{JI2017,Chen2017,jad03,mesbahi}.  Graph theoretic tools have become key method for the analysis and controller design of network systems\cite{C2015,B2007}. On one hand, classical control concepts such as stability, controllability, and observability were shown to be closely related to graphic theoretic properties; e.g., connectivity, of underlying interaction graphs \cite{Rahmani09,ji2009,Parlangeli12,Gharesifard17,Chen15}. On the other hand, distributed control synthesis provided scalable and robust control solutions without relying on centralized sensing and decisions \cite{mesbahi}.

Recently, there has been an emerging effort in bringing the graph-theoretic analysis to the study of bilinear control systems. In \cite{structural}, a framework for the structural controllability of bilinear control systems was proposed, where it was shown that the connectivity of the underlying interaction graph may determine the structural controllability of several classes of bilinear control systems over matrix groups. In \cite{li2017}, a graphical notion of permutation cycles was introduced to bilinear systems for the characterization of controllability over the special orthogonal group. These results  provided a new and interesting interpretation of the classical  Lie algebra rank condition from a graph theory perspective, and a direct benefit is a reduced computational cost in checking the controllability of high-dimensional bilinear systems. These existing results mainly focused on driftless bilinear systems.

In this paper, we study graph theoretic conditions  for the controllability and accessibility of bilinear systems, with or without drift terms. In the presence of drift, a bilinear system possesses two interaction graphs, the graph from the drift dynamics, and the graph from the controlled  dynamics. As a result, the system controllability or accessibility is determined by the coupling of the two graphs, which is difficult to understand. For bilinear systems over the special orthogonal group, the special linear group and the general linear group, we establish a systemic way of transforming the Lie bracket operations in the underlying Lie algebra, into operations of removing or creating links on the drift and controlled interaction  graphs.  As a result, we successfully establish a series of purely graphical conditions for the controllability and accessibility of such bilinear systems, where the connectivity of the union of  the drift and controlled interaction graphs turns  out to be critical. We also present constructive examples showing that these graphical conditions are in fact generically tight. Some preliminary results of the paper will be presented at the 59th IEEE Conference on Decision and Control (CDC), December 2020 \cite{wang2020}.

The remainder of the paper is organized as follows. In Section \ref{sec:pre}, we present some preliminary concepts and results for bilinear systems and graph theory. Section \ref{sec:orthogonal}, Section \ref{sec:special} and Section \ref{sec:general} present our results on bilinear systems over the special orthogonal group, the special linear group and the general linear group, respectively. Finally a few concluding remarks are given in Section \ref{sec:conc}.

\section{Preliminaries}\label{sec:pre}
\subsection{Bilinear Control Systems on Connected Lie Groups}
Let $\mathbf{G}$ be a connected Lie group and $\mathfrak{g}$ be its corresponding Lie algebra. We consider the following bilinear control system over $\mathbf{G}$:
\begin{align}\label{bilinear}
	\dot{X}(t)= \mathsf{B}_0 X(t)+\Big(\sum_{i=1}^m u_i(t) \mathsf{B}_i\Big) X(t), \ \ X(0)=X_{0},
\end{align}
where $X(t)\in \mathbf{G}, \mathsf{B}_i\in \mathfrak{g}$ for $i=0,\dots,m$, and $u_i(t) \in \mathbb{R}$ are piecewise constant control signals for $i=1,\dots,m$.
For $T\geq 0$, the set $\mathcal{R}_{T}(X_{0})$ consists of the points in $\mathbf{G}$ that are {\it attainable} from $X_{0}$ at time $T$, i.e., all terminal points $X(T)$ of solutions of system (\ref{bilinear}) originating at $X(0)=X_{0}$.

The {\it attainable set} $\mathcal{R}(X_{0})$ then is defined as the union of such sets $\mathcal{R}_{T}(X_{0})$ for all $T\geq0$;  i.e.,
$\mathcal{R}(X_{0}):=\bigcup_{T\geq0}\mathcal{R}_{T}(X_{0})\subset\mathbf{G}.$
Let $I$ be the identity of $\mathbf{G}$. From the right invariance of the system (\ref{bilinear}), it follows trivially that $\mathcal{R}(X_{0})=\mathcal{R}(I)X_{0}$ for all $X_{0}\in\mathbf{G}$.
It is easily seen that $\mathcal{R}(I)$ is a sub-semigroup of $\mathbf{G}$, which is called the system semigroup associated with (1).

\begin{definition}\label{def1}{\rm(\cite{dir08})}
	The system $(\ref{bilinear})$ is called accessible  if the semigroup $\mathcal{R}(I)$ has an interior point in $\mathbf{G}$; and controllable, if $\mathcal{R}(I)=\mathbf{G}$.
\end{definition}

The Lie algebra $\mathfrak{g}$ is a vector space that is closed under the Lie bracket operation. So if $A, B \in \mathfrak{g}$, then $[A,B]=AB-BA \in \mathfrak{g}$.
For any subset $\mathfrak{C}$ of $\mathfrak{g}$, its generated Lie subalgebra, denoted by $\mathfrak{C}_{\rm LA}$, is the smallest Lie subalgebra within $\mathfrak{g}$ that contains $\mathfrak{C}$.
The system Lie algebra of $(\ref{bilinear})$ is given as $\{\mathsf{B}_0,\mathsf{B}_1,\dots,\mathsf{B}_m\}_{\rm LA}$,
where $\{\mathsf{B}_0,\mathsf{B}_1,\dots,\mathsf{B}_m\}_{\rm LA}$ is the generated Lie subalgebra of  $\mathsf{B}_0,\mathsf{B}_1,\dots,\mathsf{B}_m$.
The algebraic criteria developed in \cite{H1972, J1975, dir08} can be used to verify the accessibility  and controllability of the system $(\ref{bilinear})$ by exploiting the algebraic structure of the system Lie algebra.

\begin{theorem}\label{basic} {\rm (See \cite{H1972,J1975,dir08})}
	$(i)$ The system $(\ref{bilinear})$ is accessible on the Lie group $\mathbf{G}$ if and only if the system Lie algebra satisfies $\{\mathsf{B}_0,\mathsf{B}_1,\dots,\mathsf{B}_m\}_{\rm LA}=\mathfrak{g}$.
	
	$(ii)$ Suppose $\mathsf{B}_0=0$ or the Lie group $\mathbf{G}$ is compact.  Then the system $(\ref{bilinear})$ is controllable on the Lie group $\mathbf{G}$ if and only if it is accessible on the Lie group $\mathbf{G}$.
\end{theorem}

\subsection{Graph Theory}
An {\em undirected} graph  $\mathrm
{G}=(\mathrm {V}, \mathrm {E})$ consists of a finite set
$\mathrm{V}$ of nodes and an edge set
$\mathrm {E}$, where  an element $e=\{i,j\}\in\mathrm {E}$ denotes   an
{\it edge}  between two distinct nodes $i\in \mathrm{V}$  and $j\in\mathrm{V}$. Two nodes $i,j\in\mathrm{E}$ are said to be {\it adjacent}
if $\{i,j\}$ is an edge  in $\mathrm{E}$. The number of adjacent nodes  of $v$ is called its {\it degree}, denoted by ${\rm deg}(v)$. A graph  $\mathrm
{G}$ is called complete if there is exactly one edge between each pair of nodes in $V$. A path between two nodes $v_1$ and $v_k$ in $\mathrm{G}$ is a sequence of distinct nodes
$
v_1v_2\dots v_{k}$
such that for any $m=1,\dots,k-1$, there is an edge between $v_m$ and $v_{m+1}$. A pair of distinct  nodes $i$ and $j$
is said to be {\it reachable} from each other if there is a path between them.  A node is always assumed to be reachable
from itself. We call graph $\mathrm{G}$ {\it connected} if every pair of distinct nodes in $\mathrm{V}$ is reachable from each other. A  subgraph of $\mathrm{G}$ associated with node set $\mathrm{V}^\ast \subseteq \mathrm{V}$, denoted as $\mathrm{G}|_{\mathrm{V}^\ast}$,
is the graph $(\mathrm{V}^\ast, \mathrm{E}^\ast)$, where $\{i,j\}\in \mathrm{E}^\ast$ if and only if $\{i,j\}\in \mathrm{E}$ for $i,j\in \mathrm{V}^\ast$.
A {\it connected component} (or just component) of  $\mathrm
{G}$ is a connected subgraph induced by some $\mathrm{V}^\ast \subseteq \mathrm{V}$, which is connected to no additional nodes in $\mathrm
{V}\setminus \mathrm{V}^\ast$. A graph $\mathrm{G}$ is a {\it bi-graph} if there is a partition of the node set into $\mathrm{V}=\mathrm{V}_{1}\bigcup \mathrm{V}_{2}$ with $\mathrm{V}_{1}$ and $\mathrm{V}_{2}$ being nonempty and mutually disjoint, where all edges are between $\mathrm{V}_{1}$ and $\mathrm{V}_{2}$.

A {\it directed} graph (digraph) $\mathcal
{G}=(\mathrm {V}, \mathcal {E})$ consists of a finite set
$\mathrm {V}$ of nodes and an arc set
$\mathcal {E}\subseteq \mathrm{V}\times\mathrm{V}$, where  $e=(i,j)\in\mathcal {E}$ denotes an
{\it arc}  from node $i\in \mathrm{V}$  to node $j\in\mathrm{V}$.
For $(i,j)\in\mathcal {E}$, we say that $i$ is an {\it in-neighbor} of $j$ and $j$ is an {\it out-neighbor} of $i$.
The number of in-neighbors and out-neighbors of $v$ is called its {\it in-degree} and {\it out-degree}, denoted as ${\rm deg}^{+}(v)$ and ${\rm deg}^{-}(v)$, respectively.
A self-loop in a digraph is an arc starting from and pointing to the same node.
A digraph $\mathcal{G}$ is {\it simple} if it has no self-loops. $\mathcal{G}$ is {\it simple complete} if $\mathcal {E}=\mathrm{V}\times\mathrm{V}\setminus \{(i,i): i\in \mathrm{V}\}$.
The digraph obtained by removing the self-loop of $\mathcal{G}$ is called the simple digraph corresponding to $\mathcal{G}$.
A directed path from a node $v_1\in \mathrm{V}$ and $v_k\in \mathrm{V}$ is a sequence of distinct nodes $v_1v_2\dots v_{k}$
such that for any $m=1,\dots,k-1$, $(v_m, v_{m+1})$ is a directed arc in $\mathcal{E}$.
We say that node $j$ is {\it reachable} from node $i$ if there is a directed path from $i$ to $j$. A digraph $\mathcal{G}$ is {\it strongly connected} if every two  nodes are mutually reachable.  A {\it weakly connected component} of a digraph $\mathcal{G}$ is a component of $\mathcal{G}$ when the directions of links are ignored.

\section{Controllability over $\SO(n)$}\label{sec:orthogonal}
The special orthogonal group,  $\SO(n)$, is  the group formed by  all $\mathbb{R}^{n\times n}$ orthogonal matrices whose determinants are equal to one. Let $E_{ij}\in\R^{n\times n}$ be the matrix with $(i,j)$-th entry being $1$ and others being $0$. Define $B_{ij}=E_{ij}-E_{ji}$. Then the set $\mathpzc{B}=\{B_{ij}:\ 1\le i<j\le n\}$ forms a basis of  the space of $n\times n$ real skew-symmetric matrices  $\so(n)$, which has the dimension $n(n-1)/2$. Clearly, $\so(n)$ is the Lie algebra of $\SO(n)$.

We consider the following bilinear system in the form of (\ref{bilinear}) that evolves over $\SO(n)$:
\begin{align}\label{eq:model1}
	\dot{X}(t)=AX(t)+\Big(\sum_{k=1}^m u_k(t)B_{i_kj_k}\Big)X(t),\ \ X(0)=I_{n},
\end{align}
where $X(t)\in \SO(n)$, $I_n$ is the $n\times n$ identity matrix, $A\in\so(n)$ is a drift dynamical term, $B_{i_kj_k}\in\mathpzc{B}$ for $k=1,\dots,m$ are controlled dynamical terms, and $u_k(t)\in\R$ is the control input which is a piecewise constant signal which multiplies $B_{i_kj_k}$ for $k=1,\dots,m$.

The system (\ref{eq:model1}) is controllable if the attainable set satisfies $\mathcal{R}(I_{n})=\SO(n)$ according to Definition \ref{def1}. Because the Lie group $\SO(n)$ is compact and connected, the system (\ref{eq:model1}) is controllable if and only if
$$\{A,B_{i_1j_i},\dots,B_{i_mj_m}\}_{\rm LA}=\so(n)$$
from Theorem \ref{basic}.

Let $\mathrm{V}=\{1,2,\dots,n\}$ be a node set. We are interested in establishing graph-theoretic conditions over the node set $\mathrm{V}$ for the controllability of the system (\ref{eq:model1}). We first introduce the following interaction graphs that naturally arise from the  drift and controlled terms $A,B_{i_1j_1},\dots,B_{i_mj_m}$ of the system (\ref{eq:model1}).

\begin{definition}\label{def2}
	$(i)$ The drift interaction  graph associated with the bilinear  system $(\ref{eq:model1})$, denoted by $\mathrm{G}_{\rm drift}$, is  defined as the undirected graph $\mathrm{G}_{\rm drift}=(\mathrm{V},\mathrm{E}_{\rm drift})$, where  $\{i,j\}\in \mathrm{E}_{\rm drift}$ if and only if $[A]_{ij}=-[A]_{ji}\neq 0$ for all $i,j\in\mathrm{V}$. \\
	$(ii)$ The controlled interaction graph associated with the bilinear  system $(\ref{eq:model1})$, denoted by $\mathrm{G}_{\rm contr}$, is  defined as the undirected graph $\mathrm{G}_{\rm contr}=(\mathrm{V},\mathrm{E}_{\rm contr})$ with $\mathrm{E}_{\rm contr}=\big\{\{i_1,j_1\},\dots, \{i_m,j_m\}\big\}$.
\end{definition}

\subsection{Main Results}
First of all, when the system \eqref{eq:model1} is driftless, i.e., $A=0$, the controllability of system \eqref{eq:model1} is entirely determined by the connectivity of $\mathrm{G}_{\rm contr}$. We present the following result.
\begin{proposition}\label{thm1}
	Suppose $A=0$. Then the system \eqref{eq:model1} is controllable on the Lie group $\SO(n)$ if and only if $\mathrm{G}_{\rm contr}$ is connected.
\end{proposition}
We would like to point out that Proposition \ref{thm1} is a more explicit form of the  Theorem 1 of \cite{zhang2016} on the same problem, where controllability was studied via permutation multiplications in a symmetric group. The connection between the graph $\mathrm{G}_{\rm contr}$ and the permutation multiplications was later noted in \cite{li2017}.  Proposition \ref{thm1} is also consistent with Theorem III.9 in \cite{structural} under the notion of structural controllability.

In the presence of the drift term $A$, it becomes extremely difficult to utilize graph-theoretic tools for studying the controllability of the system (\ref{eq:model1}). The challenge is that the system (\ref{eq:model1}) can be controlled even if $\mathrm{G}_{\rm contr}$ is disconnected with the help of $\mathrm{G}_{\rm drift}$, while the relationship between $\mathrm{G}_{\rm contr}$ and $\mathrm{G}_{\rm drift}$ in terms of the generated Lie subalgebra $\{A,B_{i_1j_1},\dots,B_{i_mj_m}\}_{\rm LA}$ is quite complex.
The following theorem gives a case in which a necessary and sufficient condition can be obtained.

\begin{theorem}\label{thm2}
	Suppose each connected component of $\mathrm{G}_{\rm contr}$ contains at least three nodes. Then the system \eqref{eq:model1} is controllable on the Lie group $\SO(n)$ if and only if the union graph $\mathrm{G}_{\rm drift} \mcup \mathrm{G}_{\rm contr}:=(\mathrm{V},\mathrm{E}_{\rm drift} \mcup \mathrm{E}_{\rm contr})$ is connected.
\end{theorem}

From Theorem \ref{thm2}, controllability verification for the system \eqref{eq:model1} from the connectivity of the graphs  $\mathrm{G}_{\rm contr}$ and $\mathrm{G}_{\rm drift}$ is still possible in the presence of drift, under an additional assumption on the minimal size of the connected components of $\mathrm{G}_{\rm contr}$.

The technical extension from  Proposition \ref{thm1} to Theorem \ref{thm2} is nontrivial, where we have to introduce a new type of graph closure operations, compared to the transitive closure operations used in the literature. The proofs of  Proposition \ref{thm1} and Theorem \ref{thm2} can be found in the appendix.

\begin{remark} 
	When the minimal three-node condition for $\mathrm{G}_{\rm contr}$'s connected components fails to hold, we can construct examples, e.g., the upcoming Example 2, where connectivity of $\mathrm{G}_{\rm drift} \mcup \mathrm{G}_{\rm contr}$ dose not guarantee controllability of the system \eqref{eq:model1}. 	It is also worth pointing out that as long as the union graph $\mathrm{G}_{\rm drift} \mcup \mathrm{G}_{\rm contr}$ is not connected, the system \eqref{eq:model1} is always uncontrollable.
\end{remark}

\subsection{Examples}
\noindent{\bf Example 1.}
Consider the system \eqref{eq:model1} evolving on $\SO(6)$. Let $A=B_{12}+2B_{13}-3B_{14}$. Let $m=4$ and $B_{i_1j_1}= B_{13}$, $B_{i_2j_2}= B_{24}$, $B_{i_3j_3}= B_{35}$, $B_{i_4j_4}= B_{46}$. The controlled interaction graph, the drift interaction graph, and their union graph, are shown, respectively,  in Figure \ref{fig1}.
\begin{figure}[H]
	\centering
	\parbox{4cm}{
		\includegraphics[width=4cm]{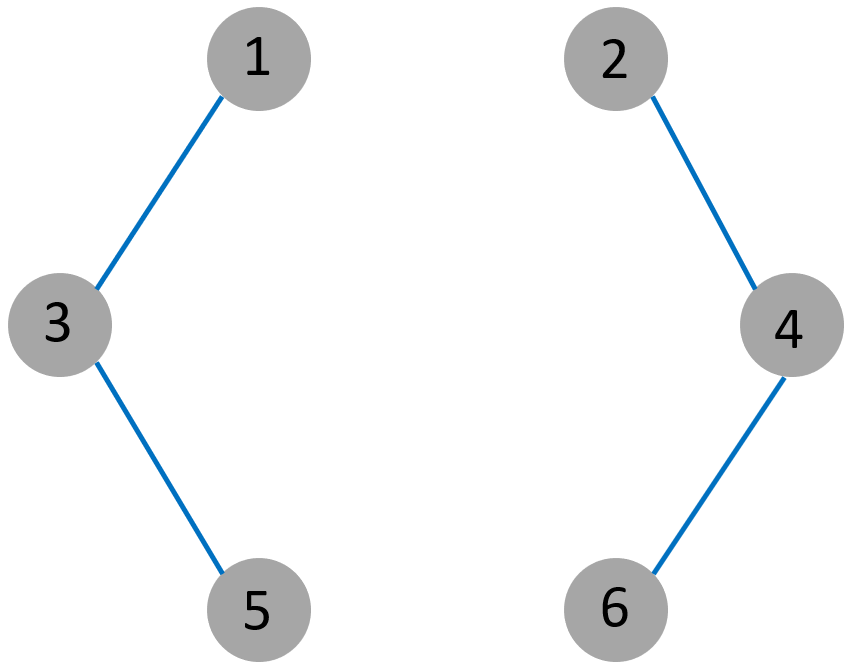}
		\subcaption{The graph $\mathrm{G}_{\rm contr}$.}}
	\qquad
	\begin{minipage}{4cm}
		\includegraphics[width=4cm]{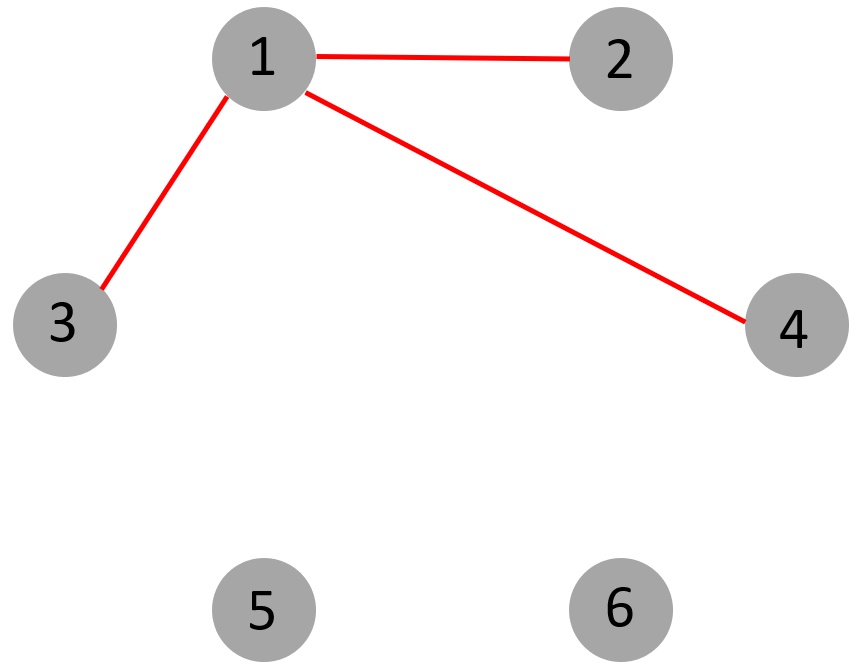}
		\subcaption{The graph $\mathrm{G}_{\rm drift}$. }
	\end{minipage}
	\qquad
	\begin{minipage}{4cm}
		\includegraphics[width=4cm]{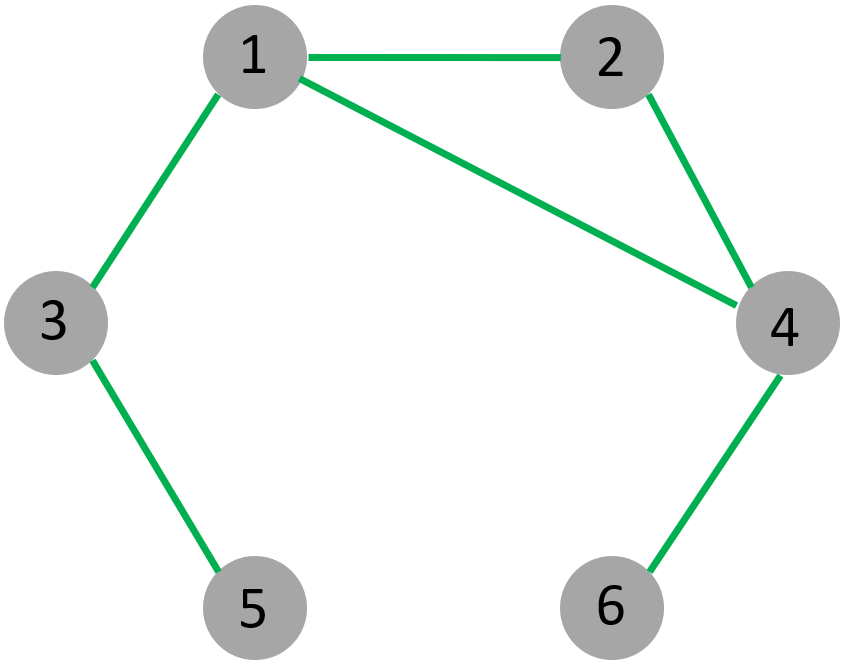}
		\subcaption{The union graph.}
	\end{minipage}
	\caption{The controlled, drift, and union interaction graphs for Example 1. }\label{fig1}
\end{figure}
It is clear that each connected component of $\mathrm{G}_{\rm contr}$ contains at least three nodes,  and the union graph $\mathrm{G}_{\rm drift} \mcup \mathrm{G}_{\rm contr}$ is connected. As a result, the graphical condition of Theorem \ref{thm2} has been met.

Since the Lie algebra is a vector space, $A-2B_{13}=B_{12}-3B_{14}\in \{A,B_{13},B_{24},B_{35},B_{46}\}_{\rm LA}$.
By direct computation one can verify $[B_{12}-3B_{14},B_{46}]=-3B_{16}$ and
$$\{A,B_{13},B_{24},B_{35},B_{46},B_{16}\}_{\rm LA}=\so(6),$$
therefore the system \eqref{eq:model1} is indeed controllable. This provides a validation of Theorem \ref{thm2}.

\noindent{\bf Example 2.}
Consider the system \eqref{eq:model1} evolving on $\SO(6)$. Let $A=B_{12}+B_{23}+B_{24}+B_{56}$. Let $m=3$ and $B_{i_1j_1}= B_{13}$, $B_{i_2j_2}= B_{24}$, and $B_{i_3j_3}= B_{46}$. The controlled interaction graph, the drift interaction graph, and their union graph, are shown, respectively,  in Figure  \ref{fig2}.
\begin{figure}[H]
	\centering
	\parbox{4cm}{
		\includegraphics[width=4cm]{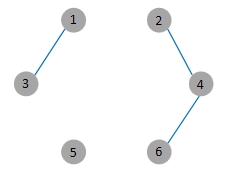}
		\subcaption{The graph $\mathrm{G}_{\rm contr}$.}}
	\qquad
	\begin{minipage}{4cm}
		\includegraphics[width=4cm]{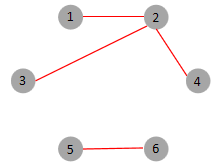}
		\subcaption{The graph $\mathrm{G}_{\rm drift}$. }
	\end{minipage}
	\qquad
	\begin{minipage}{4cm}
		\includegraphics[width=4cm]{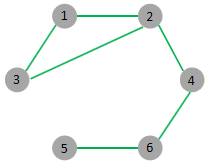}
		\subcaption{The union graph.}
	\end{minipage}
	\caption{The controlled, drift, and union interaction graphs for Example 2. }\label{fig2}
\end{figure}

It is clear that $\mathrm{G}_{\rm contr}$ has three connected components,  two of which contain less than three nodes. The union interaction graph $\mathrm{G}_{\rm drift} \mcup \mathrm{G}_{\rm contr}$ continues to be connected. By direct computation we can verify  that $\{A,B_{13},B_{24},B_{46}\}_{\rm LA}$ is a Lie subalgebra of dimension 11. Therefore, the system \eqref{eq:model1} is not controllable since $\so(6)$ is a 15-dimensional Lie algebra.

As a result, this example shows that the connectivity of the union graph $\mathrm{G}_{\rm drift} \mcup \mathrm{G}_{\rm contr}$ cannot guarantee controllability by itself. The component size condition on  $\mathrm{G}_{\rm contr}$ for Theorem \ref{thm2} is in fact rather tight.

\section{Controllability and Accessibility over ${\rm SL}(n)$}\label{sec:special}
The special linear group,  ${\rm SL}(n)$, is the group formed by all $\mathbb{R}^{n\times n}$ matrices with unit determinant. The Lie group ${\rm SL}(n)$ is connected, and its corresponding Lie algebra is the space of $n\times n$ real traceless matrices $\mathfrak{sl}(n)$, which has the dimension $n^{2}-1$. Recall that $E_{ij}\in\mathbb{R}^{n\times n}$ is the matrix with $(i,j)$-th entry being $1$ and others being $0$. Let $\mathpzc{E}_{1}=\{E_{ij}:1\le i\neq j\le n\}$, $\mathpzc{E}_{2}=\{E_{ii}:1\le i\le n\}$,
and let $\mathpzc{E}_{3}=\{C_{ij}:1\le i\neq j\le n\}$, where $C_{ij}=E_{ii}-E_{jj}$.
Then $\mathpzc{E}_{1}\mcup\mathpzc{E}_{3}$ contains a basis of $\mathfrak{sl}(n)$.

We consider the following bilinear system in the form of (\ref{bilinear}) that evolves over ${\rm SL}(n)$:
\begin{equation}\label{eq:model2}
\dot{X}(t)=AX(t)+\Big(\sum_{k=1}^{m_1} u_k(t)E_{i_kj_k}+\sum_{k=m_1+1}^{m_2}C_{i_{k}j_{k}}\Big)X(t),\ \ X(0)=I_{n},
\end{equation}
where $X(t)\in {\rm SL}(n)$, $A\in\mathfrak{sl}(n)$ is a drift dynamical term, $E_{i_kj_k}\in\mathpzc{E}_{1}$ for $k=1,\dots,m_1$, and $C_{i_kj_k}\in\mathpzc{E}_{3}$ for $k=m_1+1,\dots,m_2$ are controlled dynamical terms, $u_k(t)\in\R$ are piecewise constant functions for $k=1,\dots,m_2$.

We similarly introduce the drift and controlled interaction graphs of the system (\ref{eq:model2}), which now have to be digraphs.

\begin{definition}\label{def3}
	$(i)$ The drift interaction  graph associated with the bilinear  system {\rm(\ref{eq:model2})}, denoted by $\mathcal{G}_{\rm drift}$, is  defined as the digraph $\mathcal{G}_{\rm drift}=(\mathrm{V},\mathcal{E}_{\rm drift})$, where  $(i,j)\in \mathcal{E}_{\rm drift}$ if and only if $[A]_{ij}\neq 0$ for all $i,j\in\mathrm{V}, i\neq j$. \\
	$(ii)$ The controlled interaction graph associated with the bilinear  system {\rm(\ref{eq:model2})}, denoted by $\mathcal{G}_{\rm contr}$, is  defined as the digraph $\mathcal{G}_{\rm contr}=(\mathrm{V},\mathcal{E}_{\rm contr})$ with $\mathcal{E}_{\rm contr}=\{(i_1,j_1),\dots,(i_{m_1},j_{m_1})\}$.
\end{definition}

\begin{remark}
	Note that in system {\rm(\ref{eq:model2})} the diagonal elements of $A$, and the $C_{i_{k}j_{k}}$ terms that belong to $\mathpzc{E}_{3}$, have not been taken into consideration in $\mathcal{G}_{\rm drift}$ and $\mathcal{G}_{\rm contr}$. As it will become clear later, these terms do not contribute to the controllability and accessibility of system {\rm(\ref{eq:model2})}.
\end{remark}

\subsection{Main Results}
We present the following results on the accessibility and controllability of system \eqref{eq:model2}.
First of all, similar to the conclusion of system \eqref{eq:model1}, when the system \eqref{eq:model2} is driftless, the controllability of the system \eqref{eq:model2} is entirely determined by the strong connectivity of $\mathcal{G}_{\rm contr}$.

\begin{proposition}\label{thm3}
	Suppose $A=0$. Then the system \eqref{eq:model2} is controllable on the Lie group ${\rm SL}(n)$ if and only if the digraph  $\mathcal{G}_{\rm contr}$ is strongly connected.
\end{proposition}

In the presence of the drift term $A$, it is difficult to verify the controllability of the system \eqref{eq:model2}.
This is because Lie group ${\rm SL}(n)$ is no longer compact, and accessibility is only a necessary condition for controllability. In addition, it is hard to utilize graph-theoretic tools for the study of accessibility.
However, the following theorem gives a particular situation where a necessary and sufficient graphical condition can be used to verify accessibility.

\begin{theorem}\label{thm4}
	Suppose the weakly connected components of $\mathcal{G}_{\rm contr}$ satisfy:
	
	$(i)$ They are all strongly connected with at least two nodes;
	
	$(ii)$ One of them contains at least three nodes.
	
	Then the system \eqref{eq:model2} is accessible on ${\rm SL}(n)$ if and only if the union digraph $\mathcal{G}_{\rm drift} \mcup \mathcal{G}_{\rm contr}:=(\mathrm{V},\mathcal{E}_{\rm drift} \mcup \mathcal{E}_{\rm contr})$ is strongly connected.
\end{theorem}

Obviously, the  necessity statement of Theorem \ref{thm4} is unconditional, i.e., $\mathcal{G}_{\rm drift} \mcup \mathcal{G}_{\rm contr}$ being strongly connected is always necessary for accessibility. It can also be noticed that Theorem \ref{thm4} has a looser requirement for the minimal size of the weakly connected components of $\mathcal{G}_{\rm contr}$ than Theorem \ref{thm2}, benefiting from the properties of directed graphs. Again, this minimal size condition plays an important role in ensuring that the connectivity of the union digraph  $\mathcal{G}_{\rm drift} \mcup \mathcal{G}_{\rm contr}$ leads to the accessibility of the system \eqref{eq:model2}, which will be illustrated in  Example 4.

The proofs of Proposition \ref{thm3} and Theorem \ref{thm4} extend the analysis for Proposition \ref{thm1} and Theorem \ref{thm2} to digraphs and can be found in the appendix.

\subsection{Examples}

\noindent{\bf Example 3.} Consider the system \eqref{eq:model2} evolving on ${\rm SL}(5)$. Let $A=E_{12}+2E_{15}+E_{32}-3E_{54}+2C_{35}$. Let $m_1=5, m_2=6$, and let $E_{i_1j_1}= E_{12}$, $E_{i_2j_2}= E_{21}$, $E_{i_3j_3}= E_{54}$, $E_{i_4j_4}= E_{43}$, $E_{i_5j_5}= E_{35}$, $C_{i_6j_6}= C_{45}$. The controlled interaction digraph, the drift interaction digraph, and their union digraph are shown respectively in Figure \ref{fig3}.
\begin{figure}[H]
	\centering
	\parbox{4cm}{
		\includegraphics[width=4cm]{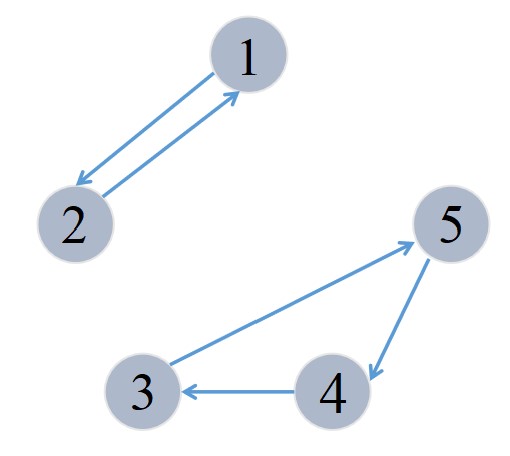}
		\subcaption{The digraph $\mathcal{G}_{\rm contr}$.}}
	\qquad
	\begin{minipage}{4cm}
		\includegraphics[width=4cm]{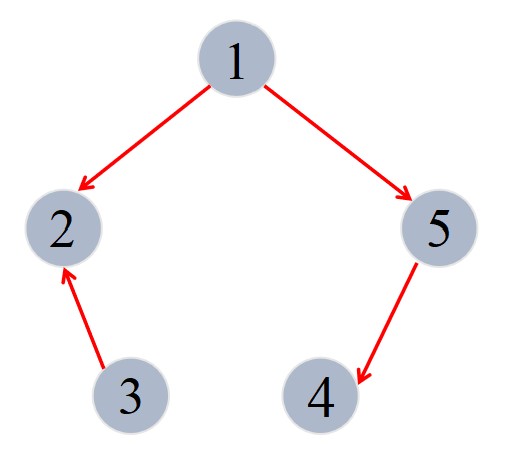}
		\subcaption{The digraph $\mathcal{G}_{\rm drift}$. }
	\end{minipage}
	\qquad
	\begin{minipage}{4cm}
		\includegraphics[width=4cm]{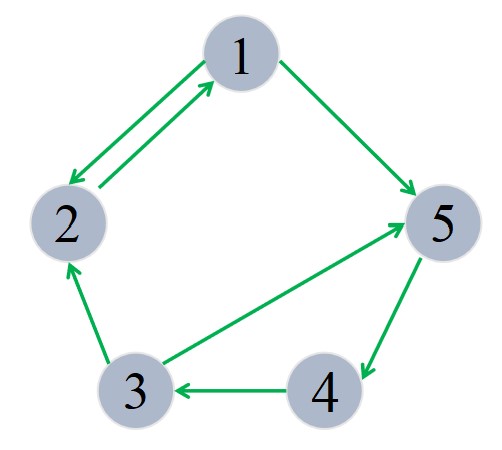}
		\subcaption{The union digraph.}
	\end{minipage}
	\caption{The controlled, drift, and union interaction digraphs for Example 3. }\label{fig3}
\end{figure}
It is clear that each weakly connected component of $\mathcal{G}_{\rm contr}$ is strongly connected with at least two nodes, and one of them contains three nodes. The union digraph $\mathcal{G}_{\rm drift} \mcup \mathcal{G}_{\rm contr}$ is strongly connected. As a result, the graphical condition of Theorem \ref{thm4} has been met.  By applying the Lie bracket repeatedly, we obtain
	$$A+2[[E_{54},E_{43}],E_{35}]-E_{12}+3E_{54}=2E_{15}+E_{32}, $$
$$	[2E_{15}+E_{32},E_{54}]=2E_{14},\ [E_{43},2E_{15}+E_{32}]=E_{42}.$$

By direct computation one can verify
$$
\big\{\{E_{12},E_{21},E_{54},E_{43},E_{35}\}\mcup\{E_{14},E_{42}\}\big\}_{\rm LA}=\mathfrak{sl}(5).
$$
Therefore, $\{A,E_{12},E_{21},E_{54},E_{43},E_{35},C_{45}\}_{\rm LA}=\mathfrak{sl}(5),$ and the system \eqref{eq:model2} is indeed accessible by Theorem \ref {basic}. This provides a validation of Theorem \ref{thm4}.

\noindent{\bf Example 4.}
Consider the system \eqref{eq:model2} evolving on ${\rm SL}(4)$. Let $A=E_{23}+E_{41}$. Let $m_1=m_2=4$  and $E_{i_1j_1}= E_{12}$, $E_{i_2j_2}= E_{21}$, $E_{i_3j_3}= E_{34}$, $E_{i_4j_4}= E_{43}$. The controlled interaction digraph, the drift interaction digraph, and their union digraph, are shown respectively in Figure \ref{fig4}.
\begin{figure}[H]
	\centering
	\parbox{4cm}{
		\includegraphics[width=4cm]{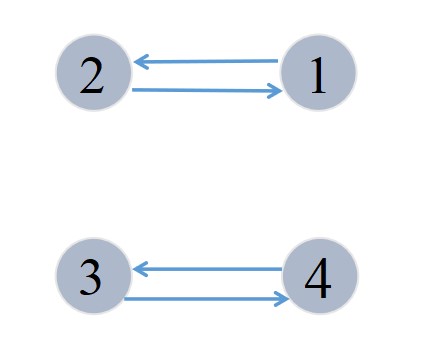}
		\subcaption{The digraph $\mathcal{G}_{\rm contr}$.}}
	\qquad
	\begin{minipage}{4cm}
		\includegraphics[width=4cm]{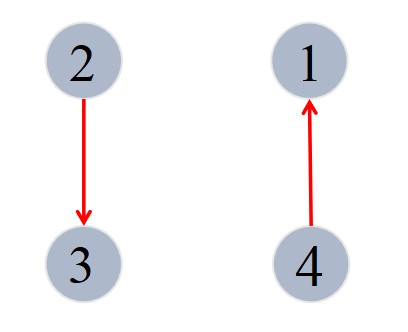}
		\subcaption{ The digraph $\mathcal{G}_{\rm drift}$.}
	\end{minipage}
	\qquad
	\begin{minipage}{4cm}
		\includegraphics[width=4cm]{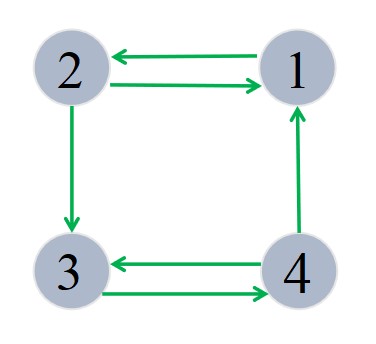}
		\subcaption{The union digraph. }
	\end{minipage}
	\caption{The controlled, drift, and union interaction digraphs for Example 4. }\label{fig4}
\end{figure}
It is clear that each weakly connected component of $\mathcal{G}_{\rm contr}$ is strongly connected with only two nodes. The union interaction digraph $\mathcal{G}_{\rm drift} \mcup \mathcal{G}_{\rm contr}$ continues to be strongly connected.  By direct computation one can verify
$$\{A,E_{12},E_{21},E_{34},E_{43}\}_{\rm LA}\neq\mathfrak{sl}(4),$$
since $\{A,E_{12},E_{21},E_{34},E_{43}\}_{\rm LA}$ is a Lie subalgebra of dimension 10, while $\mathfrak{sl}(4)$ is 15-dimensional.
Therefore, the system \eqref{eq:model2} is not accessible. This example shows that the strong connectivity of the union graph $\mathcal{G}_{\rm drift} \mcup \mathcal{G}_{\rm contr}$ cannot guarantee accessibility by itself. The weakly connected component size condition on  $\mathcal{G}_{\rm contr}$ for Theorem \ref{thm4} is in fact rather tight.

\section{Controllability and Accessibility over ${\rm GL^{+}}(n)$}\label{sec:general}
The general linear group, ${\rm GL}(n)$, is the Lie group formed by all $\mathbb{R}^{n\times n}$ invertible matrices. It has two components separated by the set of singular matrices. The connected component of ${\rm GL}(n)$ containing $I_{n}$ is a Lie subgroup of ${\rm GL}(n)$, denoted by ${\rm GL^{+}}(n)$. It is also a noncompact group. The Lie algebra associated with Lie group ${\rm GL^{+}}(n)$ is equal to $\mathfrak{gl}(n)$, the set of all $n\times n$ real matrices. Recall that $\mathpzc{E}_{1}=\{E_{ij}:1\le i\neq j\le n\}$ and $\mathpzc{E}_{2}=\{E_{ii}:1\le i\le n\}$.
Then the set $\mathpzc{E}_{1}\mcup\mathpzc{E}_{2}$ forms a standard basis of $\mathfrak{gl}(n)$, which has the dimension $n^{2}$.

We consider the following bilinear system in the form of (\ref{bilinear}) that evolves over ${\rm GL^{+}}(n)$:
\begin{align}\label{eq:model3}
	\dot{X}(t)=AX(t)+\Big(\sum_{k=1}^m u_k(t)E_{i_kj_k}\Big)X(t),\ \ X(0)=I_{n},
\end{align}
where $X(t)\in {\rm GL^{+}}(n)$, $A\in\mathfrak{gl}(n)$ is a drift dynamical term, $E_{i_kj_k}\in\mathpzc{E}_{1}\mcup\mathpzc{E}_{2}$ for $k=1,\dots,m$ are controlled dynamical terms, and $u_k(t)\in\R$ is the control input as a piecewise constant signal at the $E_{i_kj_k}$ for $k=1,\dots,m$.  Similarly, we introduce the drift and controlled interaction graphs of the system (\ref{eq:model3}).

\begin{definition}\label{def4}
	$(i)$ The drift interaction  graph associated with the bilinear  system {\rm(\ref{eq:model3})}, denoted by $\mathcal{G}_{\rm drift}^{\ast}$, is  defined as the digraph $\mathcal{G}_{\rm drift}^{\ast}=(\mathrm{V},\mathcal{E}_{\rm drift}^{\ast})$, where $(i,j)\in \mathcal{E}_{\rm drift}^{\ast}$ if and only if $[A]_{ij}\neq 0$, for $i, j\in\mathrm{V}$.\\
	$(ii)$ The controlled interaction graph associated with the bilinear  system {\rm(\ref{eq:model3})}, denoted by $\mathcal{G}_{\rm contr}^{\ast}$, is  defined as the digraph $\mathcal{G}_{\rm contr}^{\ast}=(\mathrm{V},\mathcal{E}_{\rm contr}^{\ast})$ with $\mathcal{E}_{\rm contr}^{\ast}=\big\{(i_1,j_1),\dots,(i_m,j_m)\big\}$.
\end{definition}
Note that now both $\mathcal{G}_{\rm drift}^{\ast}$ and $\mathcal{G}_{\rm contr}^{\ast}$ may have self-loops in the form of $(i,i)$, for $i\in\mathrm{V}$.

\subsection{Main Results}
We present the following results on the accessibility and controllability  of system \eqref{eq:model3}.
First of all, as opposed to the conclusion for system \eqref{eq:model2}, when the system \eqref{eq:model3} is driftless, the controllability of the system \eqref{eq:model3} is determined by the strong connectivity of $\mathcal{G}_{\rm contr}^{\ast}$ and whether there is a self-loop.

\begin{proposition}\label{thm5}
	Suppose $A=0$. Then the system \eqref{eq:model3} is controllable on the Lie group ${\rm GL^{+}}(n)$ if and only if the digraph $\mathcal{G}_{\rm contr}^{\ast}$ is a strongly connected digraph with at least one self-loop.
\end{proposition}

The extra condition on $\mathcal{G}_{\rm contr}^{\ast}$ possessing at least one self-loop cannot be removed from Proposition \ref{thm5}. Without such a condition, the terms $E_{ii}, i\in V$ will not be in the generated Lie algebra of $\{E_{i_1j_1},\dots,E_{i_mj_m}\}$. As a result, not the full $\mathfrak{gl}(n)$ can be generated, and the system connot be controllable based on Theorem 1.

In the presence of the drift term $A$, under the same conditions as Theorem \ref{thm4} of system \eqref{eq:model2},
the accessibility of the system \eqref{eq:model3} is related to the strong connectivity of $\mathcal{G}_{\rm drift}^{\ast} \mcup \mathcal{G}_{\rm contr}^{\ast}$ and whether $\mathcal{G}_{\rm contr}^{\ast}$ has a self-loop.
The following theorem shows that if the trace of matrix $A$ is not equal to zero, even if $\mathcal{G}_{\rm contr}^{\ast}$ has no self-loop, the conclusion is still true.

\begin{theorem}\label{thm6}
	Suppose the weakly connected components of $\mathcal{G}_{\rm contr}^{\ast}$  are all strongly connected with at least two nodes, and one of them contains at least three nodes. Then the system \eqref{eq:model3} is accessible on ${\rm GL^{+}}(n)$ if and only if the following conditions hold:
	
	$(i)$ The union digraph $\mathcal{G}_{\rm drift}^{\ast} \mcup \mathcal{G}_{\rm contr}^{\ast}:=(\mathrm{V},\mathcal{E}_{\rm drift}^{\ast} \mcup \mathcal{E}_{\rm contr}^{\ast})$ is strongly connected;
	
	$(ii)$  $\mathcal{G}_{\rm contr}^{\ast}$ has at least one self-loop or ${\rm tr}A\neq0$.
\end{theorem}

Note that ${\rm tr}A\neq0$ implies $A\notin \mathfrak{sl}(n).$ Then some elements in $\mathpzc{E}_{1}$ can be derived from A. Example 5 will demonstrate this process.

In particular, when $\mathcal{G}_{\rm contr}^{\ast}$ has at least one self-loop, the following theorem gives a sufficient condition for system \eqref{eq:model3} to be accessible on ${\rm GL^{+}}(n)$, which no longer requires $\mathcal{G}_{\rm contr}^{\ast}$ to have a weakly connected component containing at least three nodes.

\begin{theorem}\label{thm7}
	Suppose each weakly connected component of $\mathcal{G}_{\rm contr}^{\ast}$ is strongly connected and contains at least two nodes. Then the system \eqref{eq:model3} is accessible on ${\rm GL^{+}}(n)$ if the union digraph $\mathcal{G}_{\rm drift}^{\ast} \mcup \mathcal{G}_{\rm contr}^{\ast}:=(\mathrm{V},\mathcal{E}_{\rm drift}^{\ast} \mcup \mathcal{E}_{\rm contr}^{\ast})$ is strongly connected and $\mathcal{G}_{\rm contr}^{\ast}$ has at least one self-loop.
\end{theorem}

\begin{remark}
	In fact, for Theorem \ref{thm7}, having a weakly connected component of $\mathcal{G}_{\rm contr}^{\ast}$ contain self-loops may replace the role of the condiotion that a weakly connected component of $\mathcal{G}_{\rm contr}^{\ast}$ contains at least three nodes. 
\end{remark} 

We would like to point out that Proposition \ref {thm5}  aligns with  Theorem III.9 of \cite{structural} on a type of structural controllability, with the exception of  the self-loop requirement. Also, Theorem \ref{thm6} is related to, but significantly different from Theorem V.4 in \cite{structural} on the structural controllability of bilinear systems:

(i) In the framework of Theorem V.4, the drift term $A$
is selected from the generated subspace of $E_{i_kj_k}, k=1,\dots,m$ to ensure the accessibility.
In our framework, $A$ may have completely independent structures with respect to  $E_{i_kj_k}, k=1,\dots,m$.

(ii) From a graphical point of view, the setup of Theorem V.4 in \cite{structural} is based on the implication that $\mathcal{G}_{\rm drift}^{\ast}$ is a subgraph of the graph generated by $\mathcal{G}_{\rm contr}^{\ast}$ under the  Lie bracket operation in the algebraic space.

Besides, in the current paper controllability is defined over the groups themselves and in \cite{structural} on their actions \cite{W1975, E1979}.
\subsection{Examples}
\noindent{\bf Example 5.} Consider the system \eqref{eq:model3} evolving on ${\rm GL^{+}}(5)$.
Let $A=E_{12}+2E_{15}+E_{32}-3E_{54}+4E_{11}-E_{22}$. Let $m=5$  and $E_{i_1j_1}= E_{12}$, $E_{i_2j_2}= E_{21}$, $E_{i_3j_3}= E_{54}$, $E_{i_4j_4}= E_{43}$, $E_{i_5j_5}= E_{35}$.
The controlled interaction digraph, the drift interaction digraph, and their union digraph are shown respectively in Figure \ref{fig5}.
\begin{figure}[H]
	\centering
	\parbox{4cm}{
		\includegraphics[width=4cm]{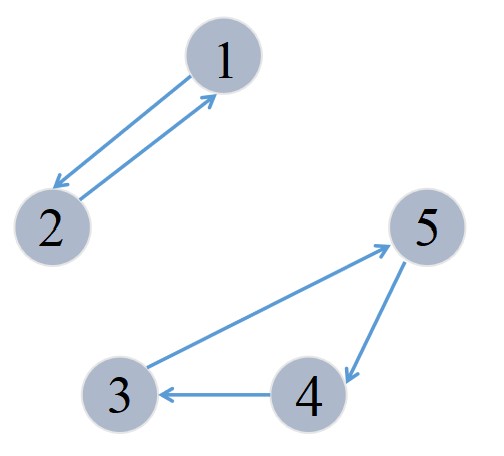}
		\subcaption{The digraph $\mathcal{G}_{\rm contr}^{\ast}$.}}
	\qquad
	\begin{minipage}{4cm}
		\includegraphics[width=4cm]{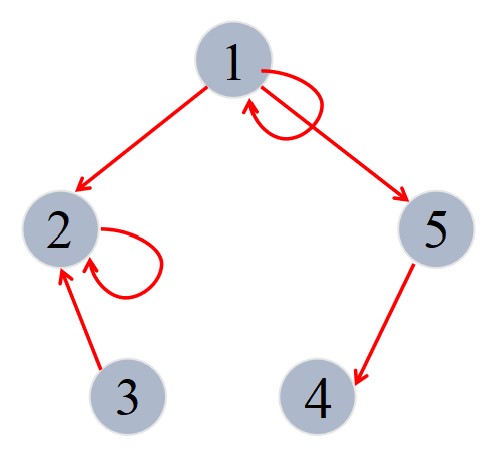}
		\subcaption{The digraph $\mathcal{G}_{\rm drift}^{\ast}$. }
	\end{minipage}
	\qquad
	\begin{minipage}{4cm}
		\includegraphics[width=4cm]{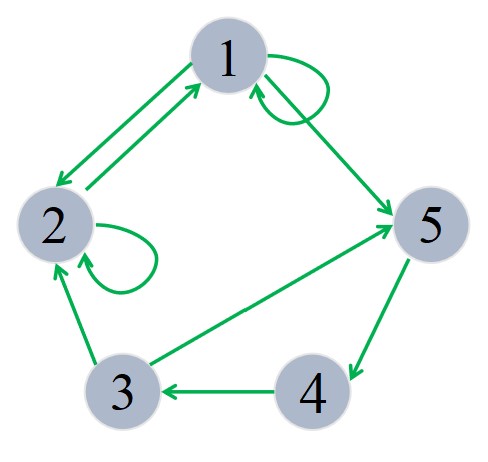}
		\subcaption{The union digraph. }
	\end{minipage}
	\caption{The controlled, drift, and union interaction digraphs for Example 5. }\label{fig5}
\end{figure}
It is clear that each weakly connected component of $\mathcal{G}_{\rm contr}^{\ast}$ is strongly connected with at least two nodes, and one of them contains at least three nodes.  The union digraph $\mathcal{G}_{\rm drift}^{\ast} \mcup \mathcal{G}_{\rm contr}^{\ast}$ is strongly connected and ${\rm tr}A=3\neq0$. As a result, the graphical condition of Theorem \ref{thm6} has been met.
Analysis similar to that in Example 3 shows that
$$\{A,E_{12},E_{21},E_{54},E_{43},E_{35}\}_{\rm LA}\supset\mathfrak{sl}(5).$$
Then $A-E_{12}-2E_{15}-E_{32}+3E_{54}=4E_{11}-E_{22}\in \{A,E_{12},E_{21},E_{54},E_{43},E_{35}\}_{\rm LA}$. Since $E_{jj}-E_{11}\in\mathfrak{sl}(5)$ for $2\leq j\leq 5$, we obtain
$(4E_{11}-E_{22})+(E_{22}-E_{11})=3E_{11}$ and
$$(E_{jj}-E_{11})+E_{11}=E_{jj}\in\{A,E_{12},E_{21},E_{54},E_{43},E_{35}\}_{\rm LA}.$$
Therefore
$$\{A,E_{12},E_{21},E_{54},E_{43},E_{35}\}_{\rm LA}=\mathfrak{gl}(5),$$
and the system \eqref{eq:model3} is indeed accessible from Theorem \ref {basic}. This provides a validation of Theorem \ref{thm6}.

\noindent{\bf Example 6.}
Consider the system \eqref{eq:model3} evolving on ${\rm GL^{+}}(4)$.
Let $A=2E_{23}-3E_{41}$. Let $m=5$  and $E_{i_1j_1}= E_{12}$, $E_{i_2j_2}= E_{21}$, $E_{i_3j_3}= E_{34}$, $E_{i_4j_4}= E_{43}$, $E_{i_5j_5}= E_{11}$ .
The controlled interaction digraph, the drift interaction digraph, and their union digraph, are shown respectively in Figure \ref{fig6}.
\begin{figure}[H]
	\centering
	\parbox{4cm}{
		\includegraphics[width=4cm]{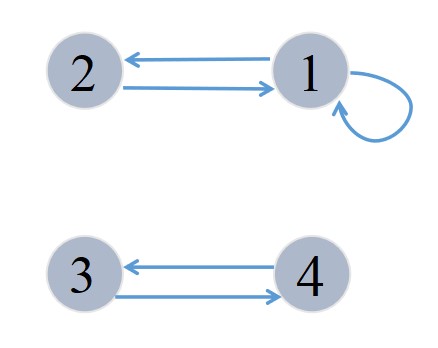}
		\subcaption{The digraph $\mathcal{G}_{\rm contr}^{\ast}$.}}
	\qquad
	\begin{minipage}{4cm}
		\includegraphics[width=4cm]{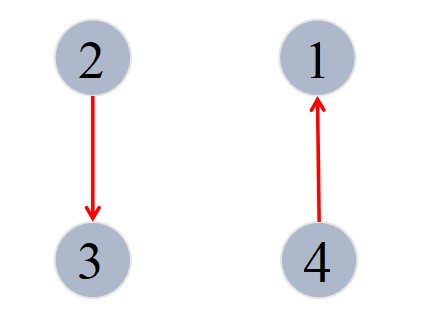}
		\subcaption{The digraph $\mathcal{G}_{\rm drift}^{\ast}$. }
	\end{minipage}
	\qquad
	\begin{minipage}{4cm}
		\includegraphics[width=4cm]{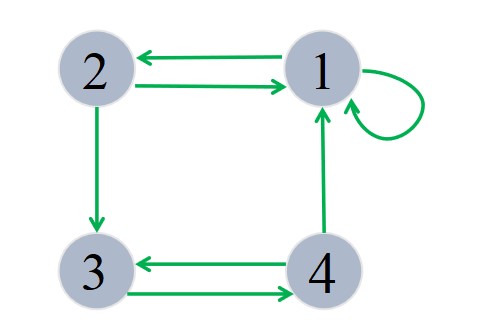}
		\subcaption{The union digraph. }
	\end{minipage}
	\caption{The controlled, drift, and union interaction digraphs for Example 6. }\label{fig6}
\end{figure}
It is clear that each weakly connected component of $\mathcal{G}_{\rm contr}^{\ast}$ is strongly connected and contains at least two nodes.  The union digraph $\mathcal{G}_{\rm drift}^{\ast} \mcup \mathcal{G}_{\rm contr}^{\ast}$ is strongly connected and $\mathcal{G}_{\rm contr}^{\ast}$ has one self-loop. As a result, the graphical condition of Theorem \ref{thm7} has been met.
By direct computation one can verify
$[A,E_{11}]=-3E_{41}, A-(-3E_{41})=2E_{23}$ and
$$
\big\{\{E_{12},E_{21},E_{34},E_{43},E_{11}\}\mcup\{E_{41},E_{23}\}\big\}_{\rm LA}=\mathfrak{gl}(4).
$$
Therefore,
$$\{A,E_{12},E_{21},E_{34},E_{43},E_{11}\}_{\rm LA}=\mathfrak{gl}(4),$$
and the system \eqref{eq:model3} is indeed accessible by Theorem \ref {basic}. This provides a validation of Theorem \ref{thm7}.

\noindent{\bf Example 7.}
Consider the system \eqref{eq:model3} evolving on ${\rm GL^{+}}(4)$.
Let $A=E_{23}+E_{41}+E_{11}+E_{33}$. Let $m=4$  and $E_{i_1j_1}= E_{12}$, $E_{i_2j_2}= E_{21}$, $E_{i_3j_3}= E_{34}$, $E_{i_4j_4}= E_{43}$.
The controlled interaction digraph, the drift interaction digraph, and their union digraph, are shown respectively in Figure \ref{fig7}.

It is clear that ${\rm tr}A\neq0$ and each weakly connected component of $\mathcal{G}_{\rm contr}^{\ast}$ is strongly connected with only two nodes. The union interaction digraph $\mathcal{G}_{\rm drift}^{\ast} \mcup \mathcal{G}_{\rm contr}^{\ast}$ continues to be strongly connected, but $\mathcal{G}_{\rm contr}^{\ast}$ has no self-loop.  By direct computation one can verify
$$\{A,E_{12},E_{21},E_{34},E_{43}\}_{\rm LA}\neq\mathfrak{gl}(4).$$
Therefore, the system \eqref{eq:model3} is not accessible.
\begin{figure}[H]
	\centering
	\parbox{4cm}{
		\includegraphics[width=4cm]{fig16}
		\subcaption{The digraph $\mathcal{G}_{\rm contr}^{\ast}$.}}
	\qquad
	\begin{minipage}{4cm}
		\includegraphics[width=4cm]{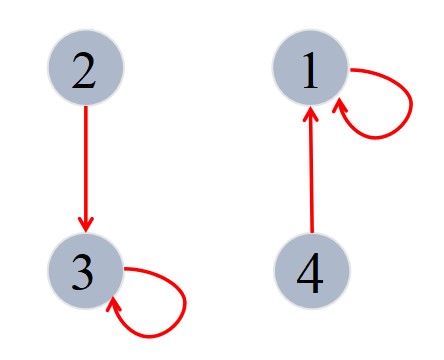}
		\subcaption{The digraph $\mathcal{G}_{\rm drift}^{\ast}$. }
	\end{minipage}
	\qquad
	\begin{minipage}{4cm}
		\includegraphics[width=4cm]{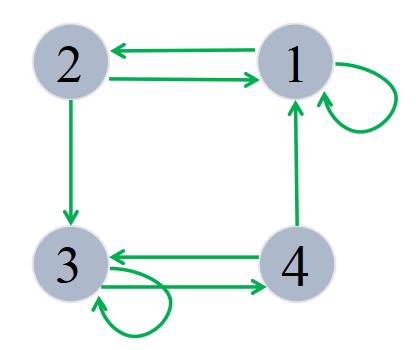}
		\subcaption{The union digraph. }
	\end{minipage}
	\caption{The controlled, drift, and union interaction digraphs for Example 7. }\label{fig7}
\end{figure}
This example shows that the strong connectivity of the union graph $\mathcal{G}_{\rm drift}^{\ast} \mcup \mathcal{G}_{\rm contr}^{\ast}$ cannot guarantee accessibility by itself even if ${\rm tr}A\neq0$. The weakly connected component size condition on $\mathcal{G}_{\rm contr}^{\ast}$ for Theorem \ref{thm6} is in fact rather tight. In addition, this example also shows that the requirement for the self-loop of $\mathcal{G}_{\rm contr}^{\ast}$ in Theorem \ref{thm7} is in fact rather tight.

\section{Conclusions}\label{sec:conc}
We have presented  graph-theoretic conditions  for the controllability and accessibility of bilinear systems over the special orthogonal group, the special linear group and the general linear group, with or without drift terms. A systemic approach was developed, for transforming the Lie bracket operations in the underlying Lie algebra space, into specific  operations of removing or creating links on the drift and controlled interaction  graphs. We established a series of purely graphical conditions on the controllability and accessibility of such bilinear systems, which rely only on the connectivity of the union of  the drift and controlled interaction graphs. Examples illustrated  that the  conditions have  in fact been  tight. In future works, it is of interest to investigate how the structure of the drift and controlled interaction graphs is related to the controllable submanifolds for the considered bilinear systems, when full controllability/accessibility is not achieved.
\section*{Appendix}
\subsection*{A. Proof of Proposition \ref{thm1}}
We first recall or establish a few auxiliary lemmas that are used in the proof.

\begin{lemma}\label{lem1}{\rm(\cite{li2017})}
	The Lie bracket of $B_{ij}$ and $B_{kl}$ in $\mathpzc{B}$ satisfies the relation
	$$
	[B_{ij},B_{kl}] = \delta_{jk}B_{il}+\delta_{il}B_{jk}+\delta_{jl}B_{ki}+\delta_{ik}B_{lj},
	$$
	where $\delta$ is the Kronecker delta function, i.e., $\delta_{mn}=1$ if $m=n$ and $\delta_{mn}=0$ otherwise.
\end{lemma}
From Lemma \ref{lem1}, we see that for any $B_{ij},B_{kl}\in\mathpzc{B}$, $[B_{ij},B_{kl}]\ne 0$ if and only if exactly one of the four equalities: $i=l$, $j=k$, $i=k$, $j=l$, holds.	Next, we introduce the following definition.

\begin{definition}\label{def5}
	Let $\mathrm{G}$ be an undirected graph. Then the graph transitive closure mapping, $\mathcal{M}(\cdot)$, is defined as
	$$
	\mathcal{M}(\mathrm{G}) = \Big(\mathrm{V}, \mathrm{E}\mcup\big\{\{i,k\}\,:\,\exists j \ \text{s.t. }\ \{i,j\}\in \mathrm{E},~\{j,k\}\in \mathrm{E}\big\}\Big).
	$$
	We then recursively define $\mathcal{M}^{k}(\mathrm{G}):=\mathcal{M}(\mathcal{M}^{k-1}(\mathrm{G}))$ with $\mathcal{M}^{1}(\mathrm{G}):=\mathcal{M}(\mathrm{G})$.
\end{definition}

\begin{lemma}\label{lem:graph.expansion}
	$\mathrm{G}$ is connected if and  only if  there exists an integer $z$ such that $\mathcal{M}^{z}(\mathrm{G})$ is a complete graph.
\end{lemma}

\begin{proof} The necessity is obvious, so we focus on the sufficiency part. Since $\mathcal{M}^{p}(\mathrm{G})$ is non-decreasing as $p$ increases, and the number of possible edges is finite, there exists an integer $z$ such that $\mathcal{M}^{p}(\mathrm{G})=\mathrm{G}_*$ for all $p\ge z$.
	
	Suppose $\mathrm{G}_*$ is not a complete graph, then there exists an edge $\{i,k\}\notin\mathrm{G}_*$. Therefore we can claim that for any $m\in\mathrm{V}\setminus\{i,k\}$, either $\{m,k\}\in\mathrm{G}_*$ or $\{m,i\}\in\mathrm{G}_*$ or $\{m,k\},\{m,i\}\notin\mathrm{G}_*$. Let $N_i:=\big\{j:\{i,j\}\in\mathrm{G}_*\big\}$ and $N_k:=\big\{j:\{k,j\}\in\mathrm{G}_*\big\}$, then from the connectivity of $\mathrm{G}$ we have $N_i\ne\emptyset$, $N_k\ne\emptyset$, and from the above claim that $N_i\cap N_k=\emptyset$.
	As a result, there must be two nodes $i_1\in N_i$, $k_1\in N_k$ such that $\{i_1,k_1\}\in\mathrm{G}_*$. Without loss of generality, we assume $i_1\ne i$ and $k_1\ne k$. From the definition of $\mathrm{G}_*$, we have $\{i,k_1\}\in\mathrm{G}_*$. Now we have a contradiction with $N_i\cap N_k=\emptyset$. The proof is complete.
\end{proof}

We are now in a position to present the detailed proof of Proposition \ref{thm1}.  Consider  $\mathscr{G}$ as the set of all undirected graphs over node set $\mathrm{V}$, and $\mathscr{B}$ as the set of all subsets of $\mathpzc{B}$. By identifying each edge $\{i,j\}$ with the matrix $B_{ij}$, we establish a natural $1-1$ correspondence, between each element in $\mathscr{G}$ and each element in $\mathscr{B}$. We denote such a mapping $\ell$ which maps from $\mathscr{B}$ to $\mathscr{G}$. For any $\mathpzc{B}_\ast \subseteq \mathpzc{B} $, we define
$$
f(\mathpzc{B}_\ast)= \mathpzc{B}_\ast \mcup \big\{[B_\alpha,B_\beta]: B_\alpha,B_\beta \in\mathpzc{B}_\ast  \big\}.
$$
From Lemma \ref{lem1}, we can conclude that
\begin{align}\label{eq1}
	\ell\big( f(\mathpzc{B}_\ast) \big) = \mathcal{M} \big(\ell(\mathpzc{B}_\ast)\big).
\end{align}

Then from  Lemma~\ref{lem:graph.expansion}, $f^z\big(\{B_{i_1j_1},\dots,B_{i_mj_m}\}\big) =\mathpzc{B}$ for some integer $z$ if and only if $\mathrm{G}_{\rm contr}$ is connected. On the other hand,  $f^z\big(\{B_{i_1j_1},\dots,B_{i_mj_m}\}\big)=\mathpzc{B}$ is equivalent to   $\{B_{i_1j_1},\dots,B_{i_mj_m}\}_{\rm LA}=\so(n)$. We have completed the proof.

\subsection*{B. Proof of Theorem \ref{thm2}}
In this subsection, we prove  Theorem \ref{thm2}.
Since $A\in\so(n)$ in system \eqref{eq:model1}, we can represent it uniquely in the form of
$$
A=\sum_{k=1}^{l} a_{k} B_{{i}^\ast_k{j}^\ast_k},
$$
where $a_k\neq 0\in\mathbb{R}$, $i^\ast_k,j^\ast_k\in\mathrm{V}$, and $B_{{i}^\ast_k{j}^\ast_k}\in\mathpzc{B}$. This gives $\mathrm{E}_{\rm drift}=\big\{\{i^\ast_1,j^\ast_1\},\dots,\{i^\ast_l,j^\ast_l\}\big\}$ by Definition \ref{def2}.

\begin{definition}\label{def6}
	Let $\mathrm{G}=(\mathrm{V},\mathrm{E})$ be an undirected graph. Given any node pair $(i,j)$ of $\mathrm{G}$, the graph
	$\mathcal{H}_{ij}(\mathrm{G}) := (\mathrm{V},\mathrm{E}_{ij})$ is called the circumjacent closure at node pair $(i,j)$ of $\mathrm{G}$ with
	$\mathrm{E}_{ij}=\mathrm{E}_{ij}^{1}\mcup\mathrm{E}_{ij}^{2}$, where
	\begin{align*}
		\mathrm{E}_{ij}^{1}&= \big\{\{i,k\}\,:\,\exists k \ {\rm s.t. }\ \{j,k\}\in \mathrm{E}\big\},  \mathrm{E}_{ij}^{2}= \big\{\{j,k\}\,:\,\exists k \ {\rm s.t. }\ \{i,k\}\in \mathrm{E}\big\}.
	\end{align*}
\end{definition}
{
	\begin{lemma}\label{lem3}
		Let $\mathrm{G}=(\mathrm{V},\mathrm{E})$ be an undirected graph. Suppose for ${i},{j}\in \mathrm{V}, \{i,j\}\notin\mathrm{E}$, and $\deg({i})=k$,\ $\deg({j})=l$. Then
		$\mathcal{H}_{{i}{j}}(\mathrm{G})= (\mathrm{V},\mathrm{E}_{{i}{j}})$ satisfies
		
		$(i)$ $|\mathrm{E}_{{i}{j}}|=k+l$, where $|\cdot|$ is the number of elements in set;
		
		$(ii)$ $\deg({i})=l$,\ $\deg({j})=k$;
		
		$(iii)$ All nodes have zero degree except for ${i}$, ${j}$ and their neighbors.
	\end{lemma}
}
\begin{proof}
	Let $i_{1},\dots,i_{k}$ and $j_{1},\dots,j_{l}$ be the nodes in $\mathrm{V}$ that are adjacent to ${i}$ and ${j}$, respectively. We can see directly by Definition \ref{def6}, that $$
	\mathrm{E}_{{i}{j}}=\big\{\{{i},j_{1}\},\dots,\{{i},j_{l}\}\big\}\mcup\big\{\{{j},i_{1}\},\dots,\{{j},i_{k}\}\big\}.
	$$
	The three statements (i), (ii), and (iii) can then be verified straightforwardly.
\end{proof}

{
	\begin{lemma}\label{lem4}
		Suppose the graph $\mathrm{G}=(\mathrm{X\mcup Y},\mathrm{E})$ is a bi-graph with $|\mathrm{X}|\geq3,|\mathrm{Y}|\geq3$ and $|\mathrm{E}|\geq1$. Then there exists  a finite sequence of node pairs $(i_1,j_1),\dots,(i_z,j_z)$ for some integer $z\geq 1$ such that
		
		$(i)$ Either $i_{s},j_{s}\in \mathrm{X}$ or $i_{s},j_{s}\in \mathrm{Y}$ for $s=1,2,\dots,z$;
		
		$(ii)$ $\mathcal{H}_{i_{z}j_{z}}\big(\cdots\mathcal{H}_{i_{2}j_{2}}(\mathcal{H}_{i_{1}j_{1}}(\mathrm{G}))\big):=(\mathrm{X\mcup Y},\mathrm{E}_{i_{z}j_{z}})$ is also a bi-graph with $|\mathrm{E}_{i_{z}j_{z}}|=1$.
	\end{lemma}
}
\begin{proof}
	It is evident that the statement holds for $|\mathrm{E}|=1$. In the rest of the proof, we focus on $|\mathrm{E}|\geq 2$.  Assume $\mathrm{X}=\{x_{1},x_{2},\dots,x_{m}\},\mathrm{Y}=\{y_{1},y_{2},\dots,y_{n}\}$ with $m\geq3$ and $n\geq3$.
	Based on whether the graph $\mathrm{G}$ contains zero degree nodes or not, our proof will be divided into two cases.
	\begin{itemize}
		\item Proof under Case (i): the graph $\mathrm{G}$ has at least one node with zero degree.
		
		Without loss of generality we assume $\deg(x_{1})=0$. Since $\sum_{i=1}^{m}\deg(x_{i})=|\mathrm{E}|\geq 2$, there exists a node $x_{i}\in \mathrm{X}, i\neq1$, such that $\deg(x_{i})=k\geq 1$. Let $y_{1},y_{2},\dots,y_{k}$ be the nodes in $\mathrm{Y}$ that are adjacent to $x_{i}$. Clearly $k\leq n$. Set $(i_1,j_1)=(x_{1},x_{i})$, and from Lemma \ref{lem3} we have
		$\mathcal{H}_{i_1j_1}(\mathrm{G}) = (\mathrm{X\mcup Y},\mathrm{E}_{i_1j_1})$ satisfying
		
		i)  $\deg(x_{1})=k$,\ $\deg(x_{2})= \deg(x_{3})=\cdots= \deg(x_{m})=0$;
		
		ii)  $\deg(y_{1})=\deg(y_{2})=\cdots= \deg(y_{k})=1$, $\deg(y_{k+1})= \deg(y_{k+2})=\cdots= \deg(y_{n})=0$;
		
		iii) $\mathrm{E}_{i_1j_1}= \big\{\{x_{1},y_{1}\},\{x_{1},y_{2}\},\dots,\{x_{1},y_{k}\}\big\}$ with $|\mathrm{E}_{i_1j_1}|=k$.
		
		So we need only one node pair $(i_1,j_1)$, i.e., $z=1$ to prove the statement for $k=1$.
		
		If $k\geq2$, then $\deg(y_{1})=1$,\ $\deg(y_{2})=1$ in graph $\mathcal{H}_{i_1j_1}(\mathrm{G})$. Now let $(i_2,j_2)=(y_{1},y_{2})$. We obtain by Lemma \ref{lem3} that
		$$\mathcal{H}_{i_2j_2}(\mathcal{H}_{i_1j_1}(\mathrm{G})) = \Big(\mathrm{X\mcup Y}, \big\{\{x_{1},y_{1}\},\{x_{1},y_{2}\}\big\}\Big).$$
		Hence, $\deg(y_{2})=1$,\ $\deg(y_{3})=0$ in graph $\mathcal{H}_{i_2j_2}(\mathcal{H}_{i_1j_1}(\mathrm{G}))$, which implies that we can select the pair of nodes $(i_3,j_3)=(y_{2},y_{3})$. As a result, we have
		$$\mathcal{H}_{i_3j_3}\big(\mathcal{H}_{i_2j_2}(\mathcal{H}_{i_1j_1}(\mathrm{G}))\big)= \Big(\mathrm{X\mcup Y}, \big\{\{x_{1},y_{3}\}\big\}\Big).$$
		Here $z=3$, and $\mathcal{H}_{i_3j_3}\big(\mathcal{H}_{i_2j_2}(\mathcal{H}_{i_1j_1}(\mathrm{G}))\big)$ is a bi-graph with $|\mathrm{E}_{i_{3}j_{3}}|=1$.
		
		\item Proof of Case (ii): the graph $\mathrm{G}$ has no node with zero degree.
		
		In this case, every node of $\mathrm{G}$ has a degree greater than or equal to one. Let $\deg(x_{1})=k$ and $\deg(x_{2})=l$, $k,l\geq 1$. Take $(i_1,j_1)=(x_{1},x_{2})$. Again by Lemma \ref{lem3} we have $\mathcal{H}_{i_1j_1}(\mathrm{G}) = (\mathrm{X\mcup Y},\mathrm{E}_{i_1j_1})$ satisfying
		
		i) $\deg(x_{1})=l$,\ $\deg(x_{2})= k$,\ $\deg(x_{3})=\dots= \deg(x_{m})=0,(m\geq3)$;
		
		ii) $|\mathrm{E}_{i_1j_1}|=k+l\geq 2$.
		
		Note that, graph $\mathcal{H}_{i_1j_1}(\mathrm{G})$ is also a bi-graph and has at least one node with zero degree. Therefore, it reduces to the Case (i).
	\end{itemize}
	We have now completed the proof of this lemma.
\end{proof}
\subsubsection*{B.1 Proof of Sufficiency for Theorem \ref{thm2}}

Let each connected component of $\mathrm{G}_{\rm contr}$ contain at least three nodes, and the union graph $\mathrm{G}_{\rm drift} \mcup \mathrm{G}_{\rm contr}:=(\mathrm{V},\mathrm{E}_{\rm drift} \mcup \mathrm{E}_{\rm contr})$ be connected. We proceed to prove the controllability of system \eqref{eq:model1}. If graph $\mathrm{G}_{\rm contr}$ is connected, then  by Theorems \ref{basic} and Proposition \ref{thm1}, the system \eqref{eq:model1} is controllable. Now we assume  graph $\mathrm{G}_{\rm contr}$ has $m$ connected components with $ m\geq2$. Let $\mathrm{G}_{\rm contr}^{i}=(\mathrm{V}_{i},\mathrm{E}_{\rm contr}^{i})$ denote the $i$-th connected component of  $\mathrm{G}_{\rm contr}$ for $i=1,\dots,m$. According to Lemma~\ref{lem:graph.expansion}, there exists an integer $z_{i}$ such that $\mathcal{M}^{z_{i}}(\mathrm{G}_{\rm contr}^{i})$ is a complete graph for each $i$. Therefore $\mathcal{M}^{z^{*}}(\mathrm{G}_{\rm contr})=\mcup_{i=1}^{m}\mathcal{M}^{z_{i}}(\mathrm{G}_{\rm contr}^{i})$, where  $z^{*}=\max_i\{z_{i}\}$. For simplicity we denote $$
\mathrm{G}_{\rm contr}^{*}=(\mathrm{V},\mathrm{E}_{\rm contr}^{*})=\mathcal{M}^{z^{*}}(\mathrm{G}_{\rm contr})
$$
with $\mathrm{E}_{\rm contr}^{*}=\mcup_{i=1}^{m}\big\{\{u,v\}:u, v\in\mathrm{V}_{i}\big\}$.
Lemma \ref{lem:graph.expansion} and (\ref{eq1}) yields that $\{B_{i_1j_1},\dots,B_{i_mj_m}\}_{\rm LA}$ is the vector space of matrices obtained by taking the span of $\{B_{ij}:\{i,j\}\in\mathrm{E}_{\rm contr}^{*}\}$.

Define $\mathrm{E}_{\rm valid}:=\mathrm{E}_{\rm drift}\setminus \mathrm{E}_{\rm contr}^{*}$ and $\mathrm{G}_{\rm valid}:= (\mathrm{V},\mathrm{E}_{\rm valid})$.
As the union graph $\mathrm{G}_{\rm drift} \mcup \mathrm{G}_{\rm contr}$ is connected while $\mathrm{G}_{\rm contr}$ is not, we always have $\mathrm{E}_{\rm valid}\neq\emptyset$; i.e., $|\mathrm{E}_{\rm valid}|\geq 1$.
Let
$$\widetilde{A}=\sum_{\{{i}^\ast_k,{j}^\ast_k\}\in\mathrm{E}_{\rm valid}}a_{k}B_{{i}^\ast_k{j}^\ast_k}.$$
To prove the controllability of system \eqref{eq:model1}, we need to consider $\{A,B_{i_1j_1},\dots,B_{i_mj_m}\}_{\rm LA}$. The following lemma states the equivalence of $\{A,B_{i_1j_1},\dots,B_{i_mj_m}\}_{\rm LA}$ and $\{\widetilde{A},B_{i_1j_1},\dots,B_{i_mj_m}\}_{\rm LA}$. This lemma is straightforward to establish from the basic properties of Lie algebras, so the detailed proof is omitted.

\begin{lemma}\label{lem5}
	The Lie algebra generated by the control and the drift terms of the bilinear system $\eqref{eq:model1}$ satisfies
	\begin{equation*}
		\{A,B_{i_1j_1},\dots,B_{i_mj_m}\}_{\rm LA}=\{\widetilde{A},B_{i_1j_1},\dots,B_{i_mj_m}\}_{\rm LA}.
	\end{equation*}
	
\end{lemma}

The remainder of the proof is divided into three steps.

\noindent{\it Step 1.} We first establish the relationship between Lie bracket $[\widetilde{A},B_{ij}]$ for $ B_{ij}\in \{B_{i_1j_1},\dots,B_{i_mj_m}\}_{\rm LA}$ and the circumjacent closure at node pair $(i,j)$ of $\mathrm{G}_{\rm valid}$. For any $D\in\so(n)$, $D$ can be represented in the form of $\sum_{k}d_{k}B_{{i}_k{j}_k}$ uniquely.
Define the function $\varphi$ that takes matrix $D\in\so(n)$ to an undirected graph $\mathrm{G}_{\rm D}:=(\mathrm{V},\mathrm{E}_{\rm D})$, where $\{i_k,j_k\}\in\mathrm{E}_{\rm D}$ if and only if $d_{k}\neq 0$. It is clear that $\varphi(\widetilde{A})=\mathrm{G}_{\rm valid}$.
From Definition \ref{def6}, there holds
\begin{equation}\label{eq2}
	\varphi\big([\widetilde{A},B_{ij}] \big)=\mathcal{H}_{ij}(\mathrm{G}_{\rm valid})
\end{equation}
for $B_{ij}\in \{B_{i_1j_1},\dots,B_{i_mj_m}\}_{\rm LA}$ and $\{i,j\}\in\mathrm{E}_{\rm contr}^{*}$.

\noindent{\it Step 2.} In this step, we prove the statement when $\mathrm{G}_{\rm contr}$ contains only two connected components. Let the graph $\mathrm{G}_{\rm contr}$ have exactly two connected components
$\mathrm{G}_{\rm contr}^{1}=(\mathrm{V}_{1},\mathrm{E}_{\rm contr}^{1})$ and $\mathrm{G}_{\rm contr}^{2}=(\mathrm{V}_{2},\mathrm{E}_{\rm contr}^{2})$. Then naturally $\mathrm{G}_{\rm valid}=(\mathrm{V}_{1}\bigcup\mathrm{V}_{2},\mathrm{E}_{\rm valid})$ is a bi-graph.
Lemma \ref{lem4} shows that there exist an integer $z\geq 1$ and a finite sequence of node pairs either $v_{s},u_{s}\in \mathrm{V}_{1}$ or $v_{s},u_{s}\in \mathrm{V}_{2}$ for $s=1,2,\dots,z$,
such that $\mathcal{H}_{v_{z}u_{z}}\big(\cdots(\mathcal{H}_{v_{1}u_{1}}(\mathrm{G}_{\rm valid}))\big)=\big(\mathrm{V}_{1}\mcup\mathrm{V}_{2},\mathrm{E}_{v_{z}u_{z}}\big)$ is also a bi-graph with $|\mathrm{E}_{v_{z}u_{z}}|=1$. Let $\mathrm{E}_{v_{z}u_{z}}=\big\{\{v^{*},u^{*}\}\big\}$, where $v^{*}\in\mathrm{V}_{1},u^{*}\in\mathrm{V}_{2}$.
Note that $\{v_{s},u_{s}\}\in\mathrm{E}_{\rm contr}^{*},$ so we have $B_{v_{s}u_{s}}\in \{B_{i_1j_1},\dots,B_{i_mj_m}\}_{\rm LA}$ for $ s=1,\dots,z$. As a result, we obtain from (\ref{eq2}) that
\begin{align*}
	\varphi\big([B_{v_{z}u_{z}},\dots,[B_{v_{2}u_{2}},[\widetilde{A},B_{v_{1}u_{1}}]]]\big)=\mathcal{H}_{v_{z}u_{z}}\big(\cdots(\mathcal{H}_{v_{1}u_{1}}(\mathrm{G}_{\rm valid}))\big).
\end{align*}
This  immediately implies
$$[B_{v_{z}u_{z}},\dots,[B_{v_{2}u_{2}},[\widetilde{A},B_{v_{1}u_{1}}]]]=a^{*}B_{v^{*}u^{*}},$$
where $a^{*}$ is the coefficient generated during the operation of the Lie brackets. It follows that $B_{v^{*}u^{*}}\in\{\widetilde{A},B_{i_1j_1},\dots,B_{i_mj_m}\}_{\rm LA}$, and thus
\begin{align}\label{eq3}
	\{\widetilde{A},B_{i_1j_1},\dots,B_{i_mj_m}\}_{\rm LA}
	=\{\widetilde{A},B_{v^{*}u^{*}},B_{i_1j_1},\dots,B_{i_mj_m}\}_{\rm LA}.
\end{align}
Because $\mathrm{G}_{\rm contr}$ has two connected components and $v^{*}\in\mathrm{V}_1,~u^{*}\in\mathrm{V}_2,$ $\ell\big(\{B_{v^{*}u^{*}},B_{i_1j_1},\dots,B_{i_mj_m}\}\big)$ is a connected graph. According to Lemma \ref{lem:graph.expansion} and (\ref{eq1}), we have
$$\{B_{v^{*}u^{*}},B_{i_1j_1},\dots,B_{i_mj_m}\}_{\rm LA}=\so(n).$$
Therefore, with Lemma \ref{lem5} and (\ref{eq3}), we can finally
conclude that
$\{A,B_{i_1j_1},\dots,B_{i_mj_m}\}_{\rm LA}=\so(n)$.
Consequently, the system \eqref{eq:model1} is controllable from Theorem \ref{basic}.

\noindent{\it Step 3}. In this step, we proceed to establish the result for the general case by induction on the number of connected components of $\mathrm{G}_{\rm contr}$.

\noindent{\it \textbf{Induction Hypothesis}}. If graph $\mathrm{G}_{\rm contr}$ contains $m\geq 2$ connected components, then the system \eqref{eq:model1} is controllable using this hypothesis, we will prove the corresponding condition $m$ replaced by $m+1$.

Recall that $\mathrm{G}_{\rm contr}^{i}=(\mathrm{V}_{i},\mathrm{E}_{\rm contr}^{i})$ is the $i$-th connected component of $\mathrm{G}_{\rm contr}$ for $i=1,\dots,m+1$.
From the definition of $\mathrm{E}_{\rm valid}$, we can deduce that $\mathrm{E}_{\rm valid}$ satisfies: i) all edges are between different $\mathrm{V}_{i}$ and no edges within each $\mathrm{V}_{i}$; ii) each $\mathrm{V}_{i}$ has at least one node with degree greater than zero, since $\mathrm{G}_{\rm drift} \mcup \mathrm{G}_{\rm contr}$ is connected.

Consider $\mathrm{G}_{\rm valid}$, we will now show that if there exist edges between nodes in $\mathrm{V}_{i}$ and nodes in $\mathrm{V}_{j}$, then a new basis element $B_{ij}\in\mathpzc{B}$, with $i\in\mathrm{V}_{i},j\in\mathrm{V}_{j}$, can be obtained by iterated Lie brackets of elements in $\{\widetilde{A},B_{i_1j_1},\dots,B_{i_mj_m}\}$.
For this purpose, we let $v_{11}$ be the node in $\mathrm{V}_{1}$ with $\deg(v_{11})=k>0$, and $v_{i_{1}j_{1}},\dots,v_{i_{k}j_{k}}$ denote the nodes in $\mcup_{i=2}^{m+1}\mathrm{V}_{i}$ that are adjacent to $v_{11}$.
For convenience, we let $v_{i_{1}j_{1}},\dots,v_{i_{r}j_{r}}$ be the nodes in $\mathrm{V}_{2}$, with $1\leq r\leq k$. For another node $v_{12}\in \mathrm{V}_{1}$, there are two possibilities for its degree: (i) $\deg(v_{12})=0$;
(ii) $\deg(v_{12})>0$.

\begin{itemize}
	\item Proof under Case (i): if $\deg(v_{12})=0$, then pick the pair of nodes $(v_{11},v_{12})$ and by Lemma \ref {lem3} we have
	$$\mathcal{H}_{v_{11}v_{12}}(\mathrm{G}_{\rm valid})=\big(\mathrm{V},\big\{\{v_{12},v_{i_{1}j_{1}}\},\dots,\{v_{12},v_{i_{k}j_{k}}\}\big\}\big).$$
	When $r\geq2$; i.e., $v_{i_{1}j_{1}},v_{i_{2}j_{2}} \in \mathrm{V}_{2}$, again from Lemma \ref {lem3}, there holds
	\begin{align*}
		\mathcal{H}_{v_{i_{1}j_{1}}v_{i_{2}j_{2}}}\big(\mathcal{H}_{v_{11}v_{12}}(\mathrm{G}_{\rm valid})\big)
		=\big(\mathrm{V},\big\{\{v_{12},v_{i_{1}j_{1}}\},\{v_{12},v_{i_{2}j_{2}}\}\big\}\big)
	\end{align*} by selecting the pair of nodes $(v_{i_{1}j_{1}},v_{i_{2}j_{2}})$.
	Because $|\mathrm{V}_{2}|\geq3$, there exists $v_{i^{*}j^{*}}\in\mathrm{V}_{2}$ such that $v_{i^{*}j^{*}}\neq v_{i_{1}j_{1}}$ and $v_{i^{*}j^{*}}\neq v_{i_{2}j_{2}}$.
	As a result, by selecting the node pair $(v_{i^{*}j^{*}}, v_{i_{1}j_{1}})$, we can obtain
	\begin{align*}
		\mathcal{H}_{v_{i^{*}j^{*}}v_{i_{1}j_{1}}}\big(\mathcal{H}_{v_{i_{1}j_{1}}v_{i_{2}j_{2}}}(\mathcal{H}_{v_{11}v_{12}}(\mathrm{G}_{\rm valid}))\big)
		=
		\big(\mathrm{V},\big\{\{v_{12},v_{i^{*}j^{*}}\}\big\}\big).
	\end{align*}
	Because $\{v_{11},v_{12}\}, \{v_{i_{1}j_{1}},v_{i_{2}j_{2}}\}, \{v_{i^{*}j^{*}}, v_{i_{1}j_{1}}\}\in\mathrm{E}_{\rm contr}^{*}$, $B_{v_{11}v_{12}},B_{v_{i_{1}j_{1}}v_{i_{2}j_{2}}},B_{v_{i^{*}j^{*}}v_{i_{1}j_{1}}} \in \{B_{i_1j_1},\dots,B_{i_mj_m}\}_{\rm LA}$.
	From this it follows that
	\begin{align*}
		\varphi\big([B_{v_{i^{*}j^{*}}v_{i_{1}j_{1}}},[B_{v_{i_{1}j_{1}}v_{i_{2}j_{2}}},[\widetilde{A},B_{v_{11}v_{12}}]]]\big)
		=\big(\mathrm{V},\big\{\{v_{12},v_{i^{*}j^{*}}\}\big\}\big).
	\end{align*}
	
	This implies
	$B_{v_{12}v_{i^{*}j^{*}}}\in\{\widetilde{A},B_{i_1j_1},\dots,B_{i_mj_m}\}_{\rm LA}$, and consequently,
	\begin{align}\label{eq4}
		\{\widetilde{A},B_{i_1j_1},\dots,B_{i_mj_m}\}_{\rm LA}
		=\{\widetilde{A},B_{v_{12}v_{i^{*}j^{*}}},B_{i_1j_1},\dots,B_{i_mj_m}\}_{\rm LA}.
	\end{align}
	
	Since node $v_{12} \in\mathrm{V}_{1}$, node $v_{i^{*}j^{*}}\in\mathrm{V}_{2}$ and $\mathrm{G}_{\rm contr}$, i.e., $\ell(\{B_{i_1,j_1},\dots,B_{i_m,j_m}\})$ has $m+1$  connected components, we have
	$\ell(\{B_{v_{12}v_{i^{*}j^{*}}},B_{i_1,j_1},\dots,B_{i_m,j_m}\})$ is a graph  with $m$ connected components.
	By the induction hypothesis, $\{\widetilde{A},B_{v_{12}v_{i^{*}j^{*}}},B_{i_1j_1},\dots,B_{i_mj_m}\}_{\rm LA}=\so(n).$ Therefore, the system (\ref{eq:model1}) is controllable by (\ref{eq4}). The same conclusion can be drawn for $r=1$, i.e., only $v_{i_{1}j_{1}}\in\mathrm{V}_{2}$. By selecting node pair $(v_{i^{*}j^{*}}, v_{i_{1}j_{1}})$, we also obtain
	$\mathcal{H}_{v_{i^{*}j^{*}}v_{i_{1}j_{1}}}(\mathcal{H}_{v_{11}v_{12}}(\mathrm{G}_{\rm valid}))
	=\big(\mathrm{V},\big\{\{v_{12},v_{i^{*}j^{*}}\}\big\}\big).$
	
	\item Proof of Case (ii): if $\deg(v_{12})=l> 0$, pick the pair of nodes $(v_{11},v_{12})$ and let $v_{i'_{1}j'_{1}},\dots,v_{i'_{l}j'_{l}}\in\mcup_{i=2}^{m+1}\mathrm{V}_{i}$ denote the neighbors of $v_{12}$. Then we obtain
	$$\mathcal{H}_{v_{11}v_{12}}(\mathrm{G}_{\rm valid})=(\mathrm{V},\mathrm{E}_{v_{11}v_{12}}),$$
	where
	\begin{align*}
		\mathrm{E}_{v_{11}v_{12}}=\big\{\{v_{12},v_{i_{1}j_{1}}\},\dots,\{v_{12},v_{i_{k}j_{k}}\}\big\}
		\mcup\big\{\{v_{11},v_{i'_{1}j'_{1}}\},\dots,\{v_{11},v_{i'_{l}j'_{l}}\}\big\}
	\end{align*}
	according to Lemma \ref{lem3}.
	Furthermore, the degree of each node in $\mathrm{V}_{1}$ of graph $\mathcal{H}_{v_{11}v_{12}}(\mathrm{G}_{\rm valid})$ is zero except for $v_{11}$ and $v_{12}$. Now $\mathrm{V}_{1}$ of graph $\mathcal{H}_{v_{11}v_{12}}(\mathrm{G}_{\rm valid})$ has at least one node with zero degree, which is due to the fact that $|\mathrm{V}_{1}|\geq3$. Hence $\mathcal{H}_{v_{11}v_{12}}(\mathrm{G}_{\rm valid})$ can be handled in the same way as shown above in Case (i), and, in consequence, the system \eqref{eq:model1} is controllable. This finishes the proof of the sufficiency.
\end{itemize}

\subsubsection*{B.2 Proof of Necessity  for Theorem \ref{thm2}}
Suppose $\mathrm{G}_{\rm contr}\mcup\mathrm{G}_{\rm drift}$ is not connected. Because the connected components are pairwise disjoint, there exist $\mathrm{V}_{i}$ and $\mathrm{V}_{j}$ such that the nodes in $\mathrm{V}_{i}$ and the nodes in $\mathrm{V}_{j}$ are not reachable from each other. Hence, the basis elements of $\so(n)$ like $B_{ij}$, with $i\in\mathrm{V}_{i}, j\in\mathrm{V}_{j}, i\neq j$, will never be in $\{\widetilde{A},B_{i_1j_1},\dots,B_{i_mj_m}\}_{\rm LA}$.
This implies $\{\widetilde{A},B_{i_1j_1},\dots,B_{i_mj_m}\}_{\rm LA}$ is a proper subset of $\so(n)$, and thus $$\{A,B_{i_1j_1},\dots,B_{i_mj_m}\}_{\rm LA}\neq\so(n)$$ by Lemma 5.
Therefore, the system \eqref{eq:model1} is not controllable by Theorem \ref{basic}, and consequently, if the system \eqref{eq:model1} is controllable, then the union graph $\mathrm{G}_{\rm contr}\mcup\mathrm{G}_{\rm drift}$ is connected.

The proof of this theorem is now completed.
\subsection*{C. Proof of Proposition \ref{thm3}}

\begin{lemma}\label{lem6}
	The Lie bracket of $E_{ij}$ and $E_{kl}$ in $\mathpzc{E}:=\mathpzc{E}_{2}\mcup\mathpzc{E}_{2}$ satisfies
	$$
	[E_{ij},E_{kl}] = \delta_{jk}E_{il}-\delta_{li}E_{kj},
	$$
	where $\delta$ is the Kronecker delta function.
\end{lemma}
Next, we introduce the transitive closure mapping for a digraph.

\begin{definition}\label{def7}
	Let $\mathcal{G}$ be a simple digraph. The simple digraph transitive closure mapping, $\mathcal{M}(\cdot)$, is defined as
	$$
	\mathcal{M}(\mathcal{G}) = \Big(\mathrm{V}, \mathcal{E}\mcup\big\{(i,k)\,:\,\exists j \ \text{s.t. }\ (i,j)\in \mathcal{E},~(j,k)\in \mathcal{E},\ i\neq k \big\}\Big).
	$$
	We then recursively define $\mathcal{M}^{k}(\mathcal{G}):=\mathcal{M}(\mathcal{M}^{k-1}(\mathcal{G}))$ with $\mathcal{M}^{1}(\mathcal{G}):=\mathcal{M}(\mathcal{G})$.
\end{definition}
An analysis similar to that in the proof of Lemma \ref{lem:graph.expansion} shows that the following lemma holds.

\begin{lemma}\label{lem7}
	Let $\mathcal{G}$ be a simple digraph. There exists an integer $z$ such that the digraph $\mathcal{M}^{z}(\mathcal{G})$ is simple complete if and only if $\mathcal{G}$ is strongly connected.
\end{lemma}

Consider  $\mathscr{G}_{1}$ as the set of all simple digraphs over the node set $\mathrm{V}$, and $\mathscr{E}_{1}$ as the set of all subsets of $\mathpzc{E}_{1}$. By identifying each arc $(i,j)$ with the matrix $E_{ij}\in\mathpzc{E}_{1} $, we establish a natural $1-1$ correspondence between each element in $\mathscr{G}_{1}$ and each element in $\mathscr{E}_{1}$. We denote such a mapping $\tilde{\ell}$ which maps from $\mathscr{E}_{1}$ to $\mathscr{G}_{1}$. For any $\mathpzc{S}_{1}\subseteq \mathpzc{E}_{1}$, we define $g(\mathpzc{S}_1)= \mathpzc{S}_1 \mcup \big\{[E_{ij},E_{jk}]=E_{ik}: E_{ij},E_{jk} \in\mathpzc{S}_1\  \textrm{and} \  i\neq k  \big\}$. Lemma \ref{lem6} and Definition \ref{def7} now lead to the following results.

\begin{lemma}\label{lem8}
	Given any $\mathpzc{S}_{1}\subseteq \mathpzc{E}_{1}, \mathpzc{S}_2 \subseteq \mathpzc{E}_{2}$, and $\mathpzc{S}_{3}\subseteq \mathpzc{E}_{3}$, we have the following conclusions:\\
	$(i)$  $\tilde{\ell}\big( g(\mathpzc{S}_1) \big) = \mathcal{M} \big(\tilde{\ell}(\mathpzc{S}_1)\big)$;\\
	$(ii)$
	For any $E_{ij}\in\mathpzc{E}_{1}$, if $E_{ij}\notin\mathpzc{S}_1$, then $E_{ij}$ can be generated by iterated Lie brackets of elements in  $\mathpzc{S}_1 \cup \mathpzc{S}_2\cup\mathpzc{S}_3$ if and only if $(i,j)$ is an arc of $\mathcal{M}^{z} \big(\tilde{\ell}(\mathpzc{S}_1)\big)$ for some integer $z$;\\
	$(iii)$
	For any $E_{ii}\in\mathpzc{E}_{2}$, if $E_{ii}\notin\mathpzc{S}_2$, then $E_{ii}$ can not be generated by iterated Lie brackets of elements in  $\mathpzc{S}_1 \cup \mathpzc{S}_2\cup\mathpzc{S}_3$.
\end{lemma}

\begin{lemma}\label{lem9}
	For any subset $\mathpzc{S}$ of $\mathpzc{E}_{1}$, if the simple digraph $\tilde{\ell}(\mathpzc{S})$ is strongly connected, then $\mathpzc{S}_{\rm LA}=\mathfrak{sl}(n).$
\end{lemma}
\begin{proof}
	Since $\tilde{\ell}(\mathpzc{S})$ is strongly connected, from Lemma~\ref{lem7} and Lemma~\ref{lem8}, we can conclude that $g^z(\mathpzc{S}) =\mathpzc{E}_{1}$ for some integer $z$. This implies that the basis elements of $\mathfrak{sl}(n)$ in $\mathpzc{E}_{1}$ can be generated by iterated Lie brackets of elements in  $\mathpzc{S}$, i.e.,
	$$\mathpzc{E}_{1}\subset\mathpzc{S}_{\rm LA}.$$
	In addition, $E_{ii}-E_{jj}\in\mathpzc{S}_{\rm LA}$ for $1\leq i\neq j\leq n$, since $E_{ij},E_{ji}\in\mathpzc{E}_{1}$ and $[E_{ij},E_{ji}]=E_{ii}-E_{jj}$.
	Therefore, all the basis elements of $\mathfrak{sl}(n)$ can be generated by iterated Lie brackets of elements in $\mathpzc{S}$, i.e., $\mathpzc{S}_{\rm LA}\supset\mathfrak{sl}(n)$, and, together with the fact that
	$\mathpzc{S}_{\rm LA}\subset\mathfrak{sl}(n)$, $\mathpzc{S}_{\rm LA}=\mathfrak{sl}(n)$ holds.
\end{proof}
We are now in a position to present the detailed proof of Proposition \ref{thm3}. Since $A=0$, Theorem \ref{basic} shows that the system (\ref{eq:model2}) is controllable on ${\rm SL}(n)$ if and only if $\{E_{i_1j_1},\dots,E_{i_{m_1}j_{m_1}},C_{i_{m_1+1}j_{m_1+1}}\dots,C_{i_{m_2}j_{m_2}}\}_{\rm LA}=\mathfrak{sl}(n).$ Therefore, the result is equivalent to showing that $\{E_{i_1j_1},\dots,E_{i_{m_1}j_{m_1}},C_{i_{m_1+1}j_{m_1+1}}\dots,C_{i_{m_2}j_{m_2}}\}_{\rm LA}=\mathfrak{sl}(n)$ if and only if $\mathcal{G}_{\rm contr}$ is strongly connected.

(Sufficiency) Suppose the simple digraph $\mathcal{G}_{\rm contr}$ associated with the bilinear system (\ref{eq:model2}) is strongly connected. It is easily seen that $\mathcal{G}_{\rm contr}=\tilde{\ell}\big(\{E_{i_1j_1},\dots,E_{i_{m_1}j_{m_1}}\}\big)$.
From Lemma~\ref{lem9}, we can conclude that $\{E_{i_1j_1},\dots,E_{i_{m_1}j_{m_1}}\}_{\rm LA}=\mathfrak{sl}(n)$.
It follows that
$$\{E_{i_1j_1},\dots,E_{i_{m_1}j_{m_1}},C_{i_{m_1+1}j_{m_1+1}}\dots,C_{i_{m_2}j_{m_2}}\}_{\rm LA}=\mathfrak{sl}(n).$$

(Necessity) Note that the nonzero matrices that can be generated by iterated Lie brackets of elements in $\mathpzc{E}_{1}\mcup\mathpzc{E}_{3}$ are either $E_{ij}$ or $E_{ii}-E_{jj}$, $i\neq j$. Moreover, there are at most $n-1$ matrices of the form $E_{ii}-E_{jj}$ that are linearly independent. Suppose the simple digraph $\mathcal{G}_{\rm contr}$ is not strongly connected. Then, for any integer $z$, there exists an arc $(i,j)\notin \mathcal{M}^{z}(\mathcal{G}_{\rm contr})$. By Lemma \ref{lem8}, we conclude that the basis elements $E_{ij}\in\mathpzc{E}_{1}$ cannot be generated by iterated Lie brackets of elements in  $\{E_{i_1j_1},\dots,E_{i_{m_1}j_{m_1}},C_{i_{m_1+1}j_{m_1+1}}\dots,C_{i_{m_2}j_{m_2}}\}$. This suggests that the dimension of $\{E_{i_1j_1},\dots,E_{i_{m_1}j_{m_1}},C_{i_{m_1+1}j_{m_1+1}}\dots,C_{i_{m_2}j_{m_2}}\}_{\rm LA}$ is at most $n^{2}-2$, which is smaller than the dimension of $\mathfrak{sl}(n)$.

Hence, $\{E_{i_1j_1},\dots,E_{i_{m_1}j_{m_1}},C_{i_{m_1+1}j_{m_1+1}}\dots,C_{i_{m_2}j_{m_2}}\}_{\rm LA}\neq\mathfrak{sl}(n)$, and consequently, we deduce that if $\{E_{i_1j_1},\dots,E_{i_{m_1}j_{m_1}},C_{i_{m_1+1}j_{m_1+1}}\dots,C_{i_{m_2}j_{m_2}}\}_{\rm LA}=\mathfrak{sl}(n)$, $\mathcal{G}_{\rm contr}$ is strongly connected.

\subsection*{D. Proof of Theorem \ref{thm4}}
Since $A\in\mathfrak{sl}(n)$ for system \eqref{eq:model2}, we can represent it in the form of
$$
A=\sum_{k=1}^{l_{1}} a_{k} E_{{i}^\ast_k{j}^\ast_k}+\sum_{k=l_{1}+1}^{l_{2}} a_{k}( E_{{i}^\ast_k{i}^\ast_k}-E_{{j}^\ast_k{j}^\ast_k})
$$
where $a_k\neq 0\in\mathbb{R}$, $i^\ast_k,j^\ast_k\in\mathrm{V}$, and $E_{{i}^\ast_k{j}^\ast_k}\in\mathpzc{E}_{1}, E_{{i}^\ast_k{i}^\ast_k}-E_{{j}^\ast_k{j}^\ast_k}\in\mathpzc{E}_{3}$. This gives $\mathcal{E}_{\rm drift}=\big\{(i^\ast_1,j^\ast_1),\dots,(i^\ast_{l_1},j^\ast_{l_1})\big\}$ by Definition \ref{def3}.

\begin{definition}\label{def8}
	Let $\mathcal{G}=(\mathrm{V},\mathcal{E})$ be a simple digraph. Given an ordered pair of nodes
	$\langle i,j\rangle$ with $i, j\in\mathrm{V}$, the digraph
	$\mathcal{H}_{ij}(\mathcal{G}) := (\mathrm{V},\mathcal{E}_{ij})$ is called the circumjacent closure at
	$\langle i,j\rangle$ of $\mathcal{G}$ with
	$
	\mathcal{E}_{ij}=\mathcal{E}_{ij}^{1}\mcup\mathcal{E}_{ij}^{2}$, where
	\begin{align*}
		\mathcal{E}_{ij}^{1}= \big\{(i,k)\,:\,\exists k\neq i \ \text{s.t. }\ (j,k)\in \mathcal{E}\big\},
		\mathcal{E}_{ij}^{2}= \big\{(k,j)\,:\,\exists k\neq j \ \text{s.t. }\ (k,i)\in \mathcal{E}\big\}.
	\end{align*}
\end{definition}

\begin{lemma}\label{lem10}
	Let $\mathcal{G}=(\mathrm{V},\mathcal{E})$ be a simple directed graph. Suppose for $i,j\in \mathrm{V}$ we have $(i,j), (j,i)\notin \mathcal{E}$ and $\deg^{+}({i})=k$,\ $\deg^{-}({j})=l$. Then
	$\mathcal{H}_{{i}{j}}(\mathcal{G})= (\mathrm{V},\mathcal{E}_{{i}{j}})$ satisfies
	
	$(i)$ $|\mathcal{E}_{{i}{j}}|=k+l$;
	
	$(ii)$ $\deg^{-}({i})=l$,\ $\deg^{+}({j})=k$;
	
	$(iii)$ all nodes have zero degree except for ${i}$, ${j}$, ${i}$'s in-neighbors and ${j}$'s out-neighbors.
\end{lemma}

\begin{proof}
	Let $i_{1},\dots,i_{k}$ and $j_{1},\dots,j_{l}$ be the nodes in $\mathrm{V}$ that are $i$'s in-neighbors and $j$'s out-neighbors, respectively. We can see directly by Definition \ref{def8}, that
	$$
	\mathcal{E}_{{i}{j}}=\big\{(i_{1},j),\dots,(i_{k},j)\big\}\mcup\big\{(i,j_{1}),\dots,(i,j_{l})\big\}.
	$$
	The three statements (i), (ii), and (iii) can then be verified straightforwardly.
\end{proof}

\subsubsection*{D.1 Proof of Sufficiency for Theorem \ref{thm4}}
Let the union graph $\mathcal{G}_{\rm drift} \mcup \mathcal{G}_{\rm contr}:=(\mathrm{V},\mathcal{E}_{\rm drift} \mcup \mathcal{E}_{\rm contr})$ be strongly connected. We proceed to prove the accessibility of system \eqref{eq:model2}. If digraph $\mathcal{G}_{\rm contr}$ is strongly connected, then  by Theorem \ref{basic} and Proposition \ref{thm3} the system \eqref{eq:model2} is accessible. Now we assume digraph $\mathcal{G}_{\rm contr}$ is the union of $m$ weakly connected components with $ m\geq2$.
Let $\mathcal{G}_{\rm contr}^{i}=(\mathrm{V}_{i},\mathcal{E}_{\rm contr}^{i})$ denote the $i$-th weakly connected component of  $\mathcal{G}_{\rm contr}$ for $i=1,\dots,m$. According to Lemma~\ref{lem7}, there exists an integer $z_{i}$ such that $\mathcal{M}^{z_{i}}(\mathcal{G}_{\rm contr}^{i})$ is a simple complete digraph for each $i$. Therefore $\mathcal{M}^{z^{*}}(\mathcal{G}_{\rm contr})=\mcup_{i=1}^{m}\mathcal{M}^{z_{i}}(\mathcal{G}_{\rm contr}^{i})$, where  $z^{*}=\max_i\{z_{i}\}$. For simplicity, we denote
$$
\mathcal{G}^{*}=(\mathrm{V},\mathcal{E}^{*})=\mathcal{M}^{z^{*}}(\mathcal{G}_{\rm contr})
$$
with $\mathcal{E}^{*}=\mcup_{i=1}^{m}\{(u,v): u, v\in\mathrm{V}_{i}, u\neq v\}$.
Recall that $\mathcal{G}_{\rm contr}=\tilde{\ell}\big( \{E_{i_1j_1},\dots,E_{i_{m_1}j_{m_1}}\}\big)$.
Lemma \ref{lem8} shows that the elements in $\{E_{ij}:(i,j)\in\mathcal{E}^{*}\}$ can be generated by iterated Lie brackets of elements in $\{E_{i_1j_1},\dots,E_{i_{m_1}j_{m_1}}\}$.

Define the digraph $\mathcal{G}_{\rm valid}$ by $\mathcal{G}_{\rm valid}:= (\mathrm{V},\mathcal{E}_{\rm valid})$, where $\mathcal{E}_{\rm valid}:=\mathcal{E}_{\rm drift}\setminus \mathcal{E}^{*}$.
As the union graph $\mathcal{G}_{\rm drift} \mcup \mathcal{G}_{\rm contr}$ is strongly connected while $\mathcal{G}_{\rm contr}$ is not, we always have $|\mathcal{E}_{\rm valid}|\geq 2$.
Let
$$\widetilde{A}=\sum_{\{{i}^\ast_k,{j}^\ast_k\}\in\mathcal{E}_{\rm valid}}a_{k}E_{{i}^\ast_k{j}^\ast_k}.$$
To prove the accessibility of system \eqref{eq:model2}, we need to consider $\{A,E_{i_1j_1},\dots,E_{i_{m_1}j_{m_1}},C_{i_{m_1+1}j_{m_1+1}}\dots,C_{i_{m_2}j_{m_2}}\}_{\rm LA}$.

\begin{lemma}\label{lem11}
	The Lie bracket of $\widetilde{A}$ and $E_{ij}\in\{E_{i_1j_1},\dots,E_{i_{m_1}j_{m_1}}\}_{\rm LA}$ satisfies
	\begin{align*}
		[\widetilde{A}, E_{ij}]\in\{A,E_{i_1j_1},\dots,E_{i_{m_1}j_{m_1}},C_{i_{m_1+1}j_{m_1+1}}\dots,C_{i_{m_2}j_{m_2}}\}_{\rm LA}.
	\end{align*}
\end{lemma}

This lemma is straightforward to establish from the basic properties of Lie algebras, so the detailed proof is omitted. Now we take four steps to complete the proof.

\noindent{\it Step 1.} We first establish the relationship between the Lie bracket $[\widetilde{A},E_{ij}]$ for $ E_{ij}\in \{E_{i_1j_1},\dots,E_{i_{m_1}j_{m_1}}\}_{\rm LA}$ and the circumjacent closure at $\langle i,j\rangle$ of $\mathcal{G}_{\rm valid}$. For any $B\in\mathfrak{sl}(n)$, $B$ can be represented in the form of $\sum_{k=1}^{l_{1}} b_{k} E_{{i}_k{j}_k}+\sum_{k=l_{1}+1}^{l_{2}} b_{k}( E_{{i}_k{i}_k}-E_{{j}_k{j}_k})$.
Define the function $\tilde{\varphi}$ that takes a matrix $B\in\mathfrak{sl}(n)$ to a simple digraph $\mathcal{G}_{B}:=(\mathrm{V},\mathcal{E}_{B})$, where $(i_k,j_k)\in\mathcal{E}_{B}$ if and only if $b_{k}\neq 0$ for $1\leq k\leq l_{1}$. It is clear that $\tilde{\varphi}(\widetilde{A})=\mathcal{G}_{\rm valid}$.
From Lemma \ref{lem6} and Definition \ref{def8}, there holds
\begin{equation}\label{eq5}
	\tilde{\varphi}\big([\widetilde{A},E_{ij}] \big)=\mathcal{H}_{ij}(\mathcal{G}_{\rm valid})
\end{equation}
for $E_{ij}\in \{E_{i_1j_1},\dots,E_{i_{m_1}j_{m_1}}\}_{\rm LA}$ and $(i,j)\in\mathcal{E}^{*}$.

\noindent{\it Step 2.}
Recall that $\mathcal{G}_{\rm contr}^{i}=(\mathrm{V}_{i},\mathcal{E}_{\rm contr}^{i})$ is the $i$-th weakly connected component of $\mathcal{G}_{\rm contr}$ for $i=1,\dots,m, m\geq2$. Let $\mathcal{G}_{\rm contr}^{1}$ be the weakly connected component that contains at least three nodes. In this step, we prove that if $\mathcal{G}_{\rm valid}$ has arcs from the nodes in $\mathrm{V}_{1}$ to the nodes in $\mathrm{V}_{k}, k\in \{2,\dots,m\}$,
then  all elements in set $\{E_{ij}:i\in\mathrm{V}_{1},~ j\in\mathrm{V}_{k}\}$ can be obtained by iterated Lie brackets of elements in $\{\widetilde{A},E_{i_1j_1},\dots,E_{i_{m_1}j_{m_1}}\}$.

For this purpose, we consider the digraph $\mathcal{G}_{\rm valid}$.
Because $\mathcal{G}_{\rm drift} \mcup \mathcal{G}_{\rm contr}$ is strongly connected, $\mathcal{G}_{\rm valid} \mcup \mathcal{G}_{\rm contr}$ is also strongly connected by the definition of $\mathcal{G}_{\rm valid}$. In addition, we can deduce that $\mathcal{G}_{\rm valid}$ satisfies: i) all arcs are between different $\mathrm{V}_{i}$ and no arcs within each $\mathrm{V}_{i}$; ii) each $\mathrm{V}_{i}$ has at least one node with out-degree greater than zero; iii) each $\mathrm{V}_{i}$ has at least one node with in-degree greater than zero.

Let $v_{11}$ denote the node in $\mathrm{V}_{1}$ with $\deg^-(v_{11})=k>0$, and $v_{i_{1}j_{1}},\dots,v_{i_{k}j_{k}}$ denote the out-neighbors of $v_{11}$.
Apparently, $v_{i_{1}j_{1}},\dots,v_{i_{k}j_{k}}\in\mathrm{V}_{2}\mcup\dots\mcup\mathrm{V}_{m}$.
For simplicity of notation we assume $\mathcal{G}_{\rm valid}$ has arcs from the nodes in $\mathrm{V}_{1}$ to the nodes in $\mathrm{V}_{2}$.
Set $v_{i_{1}j_{1}},\dots,v_{i_{r}j_{r}}$ be the nodes in $\mathrm{V}_{2}$, with $1\leq r\leq k$.  For the other node $v_{12}\in \mathrm{V}_{1}$, there are two possibilities for its in-degree: (i) $\deg^+(v_{12})=0$; (ii) $\deg^+(v_{12})> 0$.

\begin{itemize}
	\item Proof under Case (i): if $\deg^+(v_{12})=0$, then pick $\langle v_{12},v_{11}\rangle$ and by Lemma \ref {lem10} we have
	\begin{align*}
		\mathcal{H}_{v_{12}v_{11}}(\mathcal{G}_{\rm valid})
		=\big(\mathrm{V},\{(v_{12},v_{i_{1}j_{1}}),(v_{12},v_{i_{2}j_{2}}),\dots,(v_{12},v_{i_{k}j_{k}})\}\big).
	\end{align*}
	When $r\geq2$; i.e., $v_{i_{1}j_{1}},v_{i_{2}j_{2}} \in \mathrm{V}_{2}$, by selecting the node pair $\langle v_{i_{1}j_{1}},v_{i_{2}j_{2}}\rangle$ we can obtain
	\begin{equation}\label{eq6}
		\mathcal{H}_{v_{i_{1}j_{1}}v_{i_{2}j_{2}}}\big(\mathcal{H}_{v_{12}v_{11}}(\mathcal{G}_{\rm valid})\big)= \big(\mathrm{V},\{(v_{12},v_{i_{2}j_{2}})\}\big).
	\end{equation}
	Since $(v_{12},v_{11})$ and $(v_{i_{1}j_{1}},v_{i_{2}j_{2}})$ are in $\mathcal{E}^{*}$, i.e., $E_{v_{12}v_{11}},E_{v_{i_{1}j_{1}}v_{i_{2}j_{2}}}\in\{E_{i_1j_1},\dots,E_{i_{m_1}j_{m_1}}\}_{\rm LA}$, from (\ref{eq5}) and (\ref{eq6}) we conclude that
	$$
	\tilde{\varphi}\big([[\widetilde{A},E_{v_{12}v_{11}}],E_{v_{i_{1}j_{1}}v_{i_{2}j_{2}}}]\big)
	=\big(\mathrm{V},\{(v_{12},v_{i_{2}j_{2}})\}\big). $$
	This implies that
	$$[[\widetilde{A},E_{v_{12}v_{11}}],E_{v_{i_{1}j_{1}}v_{i_{2}j_{2}}}]= a^{*}E_{v_{12}v_{i_{2}j_{2}}},$$
	where $a^{*}$ is the coefficient generated during the operation of the Lie brackets.
	
	Therefore,
	$E_{v_{12}v_{i_{2}j_{2}}}\in\mathpzc{E}_{1}$ can be obtained by iterated Lie brackets of elements in $\{\widetilde{A},E_{i_1j_1},\dots,E_{i_{m_1}j_{m_1}}\}$.
	This, together with the strong connectivity of $\mathcal{G}_{\rm contr}^{1}$ and $\mathcal{G}_{\rm contr}^{2}$, implies that all elements in set  $\{E_{ij}:i\in\mathrm{V}_{1}, ~j\in\mathrm{V}_{2}\}$ can be obtained by iterated Lie brackets of elements in $\{\widetilde{A},E_{i_1j_1},\dots,E_{i_{m_1}j_{m_1}}\}$.
	
	Note that the same conclusion can be drawn for $r=1$, with $v_{i_{2}j_{2}}$ replaced by another node in $\mathrm{V}_{2}$ that is not equal to $v_{i_{1}j_{1}}$ since $|\mathrm{V}_{2}|\geq 2$.
	
	\item  Proof of Case (ii): if $\deg^+(v_{12})=l> 0$, pick $\langle v_{12},v_{11}\rangle$ and let $v_{i'_{1}j'_{1}},\dots,v_{i'_{l}j'_{l}}\in\mcup_{j=2}^{m}\mathrm{V}_{j}$ denote the in-neighbors of $v_{12}$. Then we obtain
	$$
	\mathcal{H}_{v_{12}v_{11}}(\mathcal{G}_{\rm valid})= (\mathrm{V},\mathcal{E}_{v_{12}v_{11}}),
	$$
	where
	\begin{align*}
		\mathcal{E}_{v_{12}v_{11}}=\{(v_{12},v_{i_{1}j_{1}}),(v_{12},v_{i_{2}j_{2}}),\dots,(v_{12},v_{i_{k}j_{k}})\}
		\mcup\{(v_{i'_{1}j'_{1}},v_{11}),(v_{i'_{2}j'_{2}},v_{11}),\dots,(v_{i'_{l}j'_{l}},v_{11})\}.
	\end{align*}
	It follows immediately that the digraph $\mathcal{H}_{v_{12}v_{11}}(\mathcal{G}_{\rm valid})$ satisfies:
	i) $\deg^-(v_{12})=k>0$; ii) there is a node
	$v_{13}\in \mathrm{V}_{1}$ with $\deg^+(v_{13})=0$ since $|\mathrm{V}_{1}|\geq3$.
	Then $E_{v_{13}v_{i_{2}j_{2}}}$ can be obtained by iterated Lie brackets of elements in $\{\widetilde{A},E_{i_1j_1},\dots,E_{i_{m_1}j_{m_1}}\}$
	by applying the same method as in Case (i) to the digraph $\mathcal{H}_{v_{12}v_{11}}(\mathcal{G}_{\rm valid})$.
	Hence, all elements in set  $\{E_{ij}:i\in\mathrm{V}_{1},~ j\in\mathrm{V}_{2}\}$ can be obtained by iterated Lie brackets of elements in $\{\widetilde{A},E_{i_1j_1},\dots,E_{i_{m_1}j_{m_1}}\}$ .
\end{itemize}

\noindent{\it Step 3.} In this step, we proceed to show that 
all elements in set  $\{E_{ij}:i\in\mathrm{V}_{1},~j\in\mathrm{V}_{2}\mcup\dots\mcup\mathrm{V}_{m}\}\mcup\{E_{ij}:i\in\mathrm{V}_{2}\mcup\dots\mcup\mathrm{V}_{m},~ j\in\mathrm{V}_{1}\}$ can be obtained by iterated Lie brackets of elements in $\{\widetilde{A},E_{i_1j_1},\dots,E_{i_{m_1}j_{m_1}}\}$.

We first show that all elements in set  $\{E_{ij}:i\in\mathrm{V}_{1},~j\in\mathrm{V}_{2}\mcup\dots\mcup\mathrm{V}_{m}\}$ can be obtained by iterated Lie brackets of elements in $\{\widetilde{A},E_{i_1j_1},\dots,E_{i_{m_1}j_{m_1}}\}$.
Recall that $\mathcal{G}_{\rm valid} \mcup \mathcal{G}_{\rm contr}$ is strongly connected. Graph $\mathcal{G}_{\rm valid}$ has a directed circle direction $\mathrm{V}_{j_{1}}=\mathrm{V}_{1},\mathrm{V}_{j_{2}},\dots,\mathrm{V}_{j_{k}}=\mathrm{V}_{1}$, such that there exist at least one arc from the node in $\mathrm{V}_{j_{t}}$ to the node in $\mathrm{V}_{j_{t+1}}$ for all $t=1,2,\dots,k-1$ and $\{1,2,\dots,m\}\subset\{j_{1},j_{2},\dots,j_{k}\}$.
To simplify notation, we let the directed circle direction be $\mathrm{V}_{1},\mathrm{V}_{2},\mathrm{V}_{3},\mathrm{V}_{4},\dots,\mathrm{V}_{1}$.

Since $\mathcal{G}_{\rm valid}$ has arcs from the nodes in $\mathrm{V}_{1}$ to the nodes in $\mathrm{V}_{2}$, from step 2 we know that all elements in set  $\{E_{ij}:i\in\mathrm{V}_{1},~ j\in\mathrm{V}_{2}\}$ can be obtained by iterated Lie brackets of elements in $\{\widetilde{A},E_{i_1j_1},\dots,E_{i_{m_1}j_{m_1}}\}$ .
Next, along the directed circle direction we consider $\mathrm{V}_{3}$. If there exist arcs from the nodes in $\mathrm{V}_{1}$ to the nodes in $\mathrm{V}_{3}$, obviously all elements in set  $\{E_{ij}:i\in\mathrm{V}_{1},~ j\in\mathrm{V}_{3}\}$ can be obtained by iterated Lie brackets of elements in $\{\widetilde{A},E_{i_1j_1},\dots,E_{i_{m_1}j_{m_1}}\}$, and we next consider $\mathrm{V}_{4}$ along the directed circle direction.
Otherwise, we can use $\mathrm{V}_{2}$ as a bridge to get the same conclusion since there are arcs from the nodes in $\mathrm{V}_{2}$ to the nodes in $\mathrm{V}_{3}$.

Let $v_{21}$ denote the node in $\mathrm{V}_{2}$ with $\deg^-(v_{21})=k>0$, and $v_{i_{1}j_{1}},\dots,v_{i_{k}j_{k}}$ denote the out-neighbors of $v_{21}$.
Apparently, $v_{i_{1}j_{1}},\dots,v_{i_{k}j_{k}}\in\mcup_{j=1,j\neq2}^{m}\mathrm{V}_{j}$. For convenience we let $v_{i_{1}j_{1}},\dots,v_{i_{r}j_{r}}$ be the nodes in $\mathrm{V}_{3}$, with $1\leq r\leq k$. Similarly, for the other node $v_{22}\in \mathrm{V}_{2}$, there are two possibilities for its in-degree: (i) $\deg^+(v_{22})=0$; (ii) $\deg^+(v_{22})>0$.

\begin{itemize}
	\item Proof under Case (i): if $\deg^+(v_{22})=0$, then pick $\langle v_{22},v_{21}\rangle$ and by Lemma \ref {lem10} we have
	\begin{align*}
		\mathcal{H}_{v_{22}v_{21}}(\mathcal{G}_{\rm valid})
		=\big(\mathrm{V},\{(v_{22},v_{i_{1}j_{1}}),(v_{22},v_{i_{2}j_{2}}),\dots,(v_{22},v_{i_{k}j_{k}})\}\big).
	\end{align*}
	When $r\geq2$; i.e., $v_{i_{1}j_{1}},v_{i_{2}j_{2}} \in \mathrm{V}_{3}$, again from Lemma \ref {lem10}, there holds
	\begin{equation*}
		\mathcal{H}_{v_{i_{1}j_{1}}v_{i_{2}j_{2}}}\big(\mathcal{H}_{v_{22}v_{21}}(\mathcal{G}_{\rm valid})\big)= \big(\mathrm{V},\{(v_{22},v_{i_{2}j_{2}})\}\big).
	\end{equation*}
	This implies that
	$E_{v_{22}v_{i_{2}j_{2}}}$ can be obtained by iterated Lie brackets of elements in
	$ \{\widetilde{A},E_{i_1j_1},\dots,E_{i_{m_1}j_{m_1}}\}$.
	Note that the same conclusion can be drawn for $r=1$, with $v_{i_{2}j_{2}}$ replaced by another node in $\mathrm{V}_{3}$ that is not equal to $v_{i_{1}j_{1}}$ since $|\mathrm{V}_{3}|\geq 2$.
	
	Hence, all elements in set  $\{E_{ij}:i\in\mathrm{V}_{2},~ j\in\mathrm{V}_{3}\}$ can be obtained by iterated Lie brackets of elements in $\{\widetilde{A},E_{i_1j_1},\dots,E_{i_{m_1}j_{m_1}}\}$. In addition, We already know that all elements in set  $\{E_{ij}:i\in\mathrm{V}_{1},~ j\in\mathrm{V}_{2}\}$ can be obtained by iterated Lie brackets of elements in $\{\widetilde{A},E_{i_1j_1},\dots,E_{i_{m_1}j_{m_1}}\}$. 
	Consequently, all elements in set  $\{E_{ij}:i\in\mathrm{V}_{1},~ j\in\mathrm{V}_{3}\}$ can be obtained by iterated Lie brackets of elements in $\{\widetilde{A},E_{i_1j_1},\dots,E_{i_{m_1}j_{m_1}}\}$. 
	
	\item Proof of Case (ii): if $\deg^+(v_{22})=l> 0$, pick $\langle v_{22},v_{21}\rangle$ and let $v_{i'_{1}j'_{1}},\dots,v_{i'_{l}j'_{l}}\in\mcup_{j=1,j\neq2}^{m}\mathrm{V}_{j}$ denote the in-neighbors of $v_{22}$. Then we obtain
	$$
	\mathcal{H}_{v_{22}v_{21}}(\mathcal{G}_{\rm valid})= (\mathrm{V},\mathcal{E}_{v_{22}v_{21}}),
	$$
	where
	\begin{align*}
		\mathcal{E}_{v_{22}v_{21}}=\{(v_{22},v_{i_{1}j_{1}}),(v_{22},v_{i_{2}j_{2}}),\dots,(v_{22},v_{i_{k}j_{k}})\}
		\mcup\{(v_{i'_{1}j'_{1}},v_{21}),(v_{i'_{2}j'_{2}},v_{21}),\dots,(v_{i'_{l}j'_{l}},v_{21})\}.
	\end{align*}
	
	When $r\geq2$, i.e., $v_{i_{1}j_{1}},v_{i_{2}j_{2}} \in \mathrm{V}_{3}$, and $v_{i_{2}j_{2}} \in \{v_{i'_{1}j'_{1}},\dots,v_{i'_{l}j'_{l}}\}$.
	There holds
	\begin{align*}
		\mathcal{H}_{v_{i_{1}j_{1}}v_{i_{2}j_{2}}}\big(\mathcal{H}_{v_{22}v_{21}}(\mathcal{G}_{\rm valid})\big)= \big(\mathrm{V},\{(v_{22},v_{i_{2}j_{2}}),(v_{i_{1}j_{1}},v_{21})\}\big).
	\end{align*}
	It follows immediately that the digraph $\mathcal{H}_{v_{i_{1}j_{1}}v_{i_{2}j_{2}}}\big(\mathcal{H}_{v_{22}v_{21}}(\mathcal{G}_{\rm valid})\big)$ satisfies:
	i) $\deg^-(v_{22})=1>0$; ii) there is a node
	$v_{11}\in \mathrm{V}_{1}$ with $\deg^+(v_{11})=0$.
	Therefore,
	\begin{align*}
		\mathcal{H}_{v_{11}v_{22}}\big(\mathcal{H}_{v_{i_{1}j_{1}}v_{i_{2}j_{2}}}\big(\mathcal{H}_{v_{22}v_{21}}(\mathcal{G}_{\rm valid})\big)\big)
		= \big(\mathrm{V},
		\{(v_{11},v_{i_{2}j_{2}})\}\big).
	\end{align*}
	This shows that
	$E_{v_{11}v_{i_{2}j_{2}}}$ can be obtained by iterated Lie brackets of elements in
	$ \{\widetilde{A},E_{i_1j_1},\dots,E_{i_{m_1}j_{m_1}}\}$.
	Note that $v_{11}\in\mathrm{V}_{1}, v_{i_{2}j_{2}}\in \mathrm{V}_{3}$, which is the desired result.
	
	When $r\geq2$, i.e., $v_{i_{1}j_{1}},v_{i_{2}j_{2}} \in \mathrm{V}_{3}$, and $v_{i_{2}j_{2}} \notin \{v_{i'_{1}j'_{1}},\dots,v_{i'_{l}j'_{l}}\}$.
	There holds
	\begin{equation*}
		\mathcal{H}_{v_{i_{1}j_{1}}v_{i_{2}j_{2}}}\big(\mathcal{H}_{v_{22}v_{21}}(\mathcal{G}_{\rm valid})\big)= \big(\mathrm{V},\{(v_{22},v_{i_{2}j_{2}})\}\big).
	\end{equation*}
	According to Case (i), all elements in set  $\{E_{ij}:i\in\mathrm{V}_{1},~ j\in\mathrm{V}_{3}\}$ can be obtained by iterated Lie brackets of elements in $\{\widetilde{A},E_{i_1j_1},\dots,E_{i_{m_1}j_{m_1}}\}$.
	
	Note that the same conclusion can be drawn for $r=1$, with $v_{i_{2}j_{2}}$ replaced by another node in $\mathrm{V}_{3}$ that is not equal to $v_{i_{1}j_{1}}$ since $|\mathrm{V}_{3}|\geq 2$.
\end{itemize}

Repeated application of this method enables us to get all elements in set  $\{E_{ij}:i\in\mathrm{V}_{1},~ j\in\mathrm{V}_{2}\mcup\cdots\mcup\mathrm{V}_{m}\}$ by iterated Lie brackets of elements in $\{\widetilde{A},E_{i_1j_1},\dots,E_{i_{m_1}j_{m_1}}\}$. 
On the other hand, we continue in this fashion to obtain all elements in set  $\{E_{ij}:i\in\mathrm{V}_{2}\mcup\cdots\mcup\mathrm{V}_{m},~ j\in\mathrm{V}_{1}\}$ by considering the node in $\mathrm{V}_{1}$ whose in-degree is greater than zero.

An analysis similar to the above shows that 
all elements in set  $\{E_{ij}:i\in\mathrm{V}_{j_{k-1}},~ j\in\mathrm{V}_{1}\}$ can be obtained by iterated Lie brackets of elements in $\{\widetilde{A},E_{i_1j_1},\dots,E_{i_{m_1}j_{m_1}}\}$.
Then we can use $\mathrm{V}_{j_{k-1}}$ as a bridge to get all elements in set $\{E_{ij}:i\in\mathrm{V}_{j_{k-2}},~ j\in\mathrm{V}_{1}\}$.
And finally we conclude that all elements in set $\{E_{ij}:i\in\mathrm{V}_{2}\mcup\cdots\mcup\mathrm{V}_{m},~ j\in\mathrm{V}_{1}\}$  can be obtained by iterated Lie brackets of elements in $\{\widetilde{A},E_{i_1j_1},\dots,E_{i_{m_1}j_{m_1}}\}$.

\noindent{\it Step 4.} In this step, we prove $\{A,E_{i_1j_1},\dots,E_{i_{m_1}j_{m_1}}$,   
$C_{i_{m_1+1}j_{m_1+1}}\dots,C_{i_{m_2}j_{m_2}}\}_{\rm LA}=\mathfrak{sl}(n)$ and complete the proof.

Let $S$ denote the union of $\{E_{ij}:i\in\mathrm{V}_{1},~ j\in\mathrm{V}_{2}\mcup\cdots\mcup\mathrm{V}_{m}\}$ and $\{E_{ij}:i\in\mathrm{V}_{2}\mcup\cdots\mcup\mathrm{V}_{m},~ j\in\mathrm{V}_{1}\}$. From step 3 we see that all elements in set $S$  can be obtained by iterated Lie brackets of elements in $\{\widetilde{A},E_{i_1j_1},\dots,E_{i_{m_1}j_{m_1}}\}$.
This, together with Lemma \ref{lem11}, implies that $S$ is a subset of $\{A,E_{i_1j_1},\dots,E_{i_{m_1}j_{m_1}}$,   
$C_{i_{m_1+1}j_{m_1+1}}\dots,C_{i_{m_2}j_{m_2}}\}_{\rm LA}$. Consider the connectivity of digroup $\tilde{\ell}\big(S\mcup\{E_{i_1j_1},\dots,E_{i_{m_1}j_{m_1}}\}\big)$. Note that for any $\mathrm{V}_{j}, ~j=2,\dots,m$, there exist $E_{ij}$ and $E_{ji}$ in $S$ with $i\in \mathrm{V}_{1}, j\in\mathrm{V}_{j}$. Hence, all of the nodes in $\mathrm{V}_{1}\mcup\mathrm{V}_{j}$ are on a directed circle. This is due to the fact that each $\mathcal{G}_{\rm contr}^{i}$ is strongly connected. Therefore, digraph $\tilde{\ell}\big(S\mcup\{E_{i_1j_1},\dots,E_{i_{m_1}j_{m_1}}\}\big)$ has a directed circle containing all the nodes in $\mathrm{V}$. This implies that $\tilde{\ell}\big(S\mcup\{E_{i_1j_1},\dots,E_{i_{m_1}j_{m_1}}\}\big)$ is strongly connected. From Lemma \ref{lem9} we have
$$\big\{S\mcup\{E_{i_1j_1},\dots,E_{i_{m_1}j_{m_1}}\}\big\}_{\rm LA}=\mathfrak{sl}(n).$$

Therefore
$$\{A,E_{i_1j_1},\dots,E_{i_{m_1}j_{m_1}},C_{i_{m_1+1}j_{m_1+1}}\dots,C_{i_{m_2}j_{m_2}}\}_{\rm LA}=\mathfrak{sl}(n),$$
and consequently, the system \eqref{eq:model2} is accessible on ${\rm SL}(n)$ by Theorem \ref{basic}.
The proof of sufficiency is now completed.

\subsubsection*{D.2 Proof of Necessity  for Theorem \ref{thm4}}
Suppose $\mathcal{G}_{\rm contr}\mcup\mathcal{G}_{\rm drift}$ is not strongly connected. Then there exist $\mathrm{V}_{i}$ and $\mathrm{V}_{j}$ such that every node in $\mathrm{V}_{j}$ is not reachable from node in $\mathrm{V}_{i}$. Hence, the basis elements of $\mathfrak{sl}(n)$ in  $\{E_{kl}:k\in\mathrm{V}_{i}, ~l\in\mathrm{V}_{j}\}$ can never be obtained by iterated Lie brackets of elements in $ \{\widetilde{A},E_{i_1j_1},\dots,E_{i_{m_1}j_{m_1}}\}$.
This, together with Lemma \ref{lem8} (ii), implies that $\{A,E_{i_1j_1},\dots,E_{i_{m_1}j_{m_1}},C_{i_{m_1+1}j_{m_1+1}}\dots,C_{i_{m_2}j_{m_2}}\}_{\rm LA}$ is a proper subset of $\mathfrak{sl}(n)$, i.e.,
$$\{A,E_{i_1j_1},\dots,E_{i_{m_1}j_{m_1}},C_{i_{m_1+1}j_{m_1+1}}\dots,C_{i_{m_2}j_{m_2}}\}_{\rm LA}\neq\mathfrak{sl}(n).$$
Therefore, the system \eqref{eq:model2} is not accessible, and consequently, if the system \eqref{eq:model2} is accessible, the union graph $\mathcal{G}_{\rm contr}\mcup\mathcal{G}_{\rm drift}$ is strongly connected.
This finishes the proof of this theorem.

\subsection*{E. Proof of Proposition \ref{thm5}}
Since $A=0$, Theorem \ref{basic} shows that the system (\ref{eq:model3}) is controllable on ${\rm GL^{+}}(n)$ if and only if $\{E_{i_1j_1},\dots,$ $E_{i_mj_m}\}_{\rm LA}=\mathfrak{gl}(n).$ Therefore, the result is equivalent to showing that $\{E_{i_1j_1},\dots,E_{i_mj_m}\}_{\rm LA}=\mathfrak{gl}(n)$ if and only if $\mathcal{G}_{\rm contr}^{\ast}$ is a strongly connected digraph with at least one self-loop.

(Sufficiency) Suppose the digraph $\mathcal{G}_{\rm contr}^{\ast}$ associated with the bilinear system (\ref{eq:model3}) is strongly connected and has at least one self-loop.
It follows easily  that the simple digraph corresponding to $\mathcal{G}_{\rm contr}^{\ast}$ is still strongly connected. Lemma \ref{lem9} now shows that all the elements in $\mathpzc{E}_{1}\mcup\mathpzc{E}_{3}$ can be generated by iterated Lie brackets of elements in $\{E_{i_1j_1},\dots,E_{i_mj_m}\}$.

Let $(i,i)$ be a self-loop of $\mathcal{G}_{\rm contr}^{\ast}$, i.e., $E_{ii}\in\{E_{i_1j_1},\dots,E_{i_mj_m}\}$.
Because the Lie algebra is a vector space, we have $(E_{jj}-E_{ii})+E_{ii}\in\{E_{i_1j_1},\dots,E_{i_mj_m}\}_{\rm LA}$, which leads to
$E_{jj}\in\{E_{i_1j_1},\dots,E_{i_mj_m}\}_{\rm LA}$
for any
$1\le j\le n, j\neq i$.
Therefore, all the basis elements in $\mathpzc{E}_{1}\mcup\mathpzc{E}_{2}$ of $\mathfrak{gl}(n)$ are in the Lie algebra generated by  $\{E_{i_1j_1},\dots,E_{i_mj_m}\}$. Consequently,
$$\{E_{i_1j_1},\dots,E_{i_mj_m}\}_{\rm LA}=\mathfrak{gl}(n).$$

(Necessity)
Suppose the digraph $\mathcal{G}_{\rm contr}^{\ast}$ is not strongly connected or has no self-loops.
If $\mathcal{G}_{\rm contr}^{\ast}$ is not strongly connected, then its corresponding simple digraph is not either.
By Lemma \ref{lem7} and Lemma \ref{lem8}, we conclude that there exists at least one basis element $E_{ij}\in\mathpzc{E}_{1}$ that cannot be generated by iterated Lie brackets of elements in  $\{E_{i_1j_1},\dots,E_{i_mj_m}\}$. This implies that the dimension of $\{E_{i_1j_1},\dots,E_{i_mj_m}\}_{\rm LA}$ is at most $n^{2}-1$.

If $\mathcal{G}_{\rm contr}^{\ast}$ does not contain self-loops, Lemma \ref{lem8} shows that for any $i=1,\dots,n$ , $E_{ii}\in\mathpzc{E}_{2}$ will never be in $\{E_{i_1j_1},\dots,E_{i_mj_m}\}_{\rm LA}, i=1,\cdots,n.$ Therefore, using Lemma \ref{lem9}, we deduce that the dimension of $\{E_{i_1j_1},\dots,E_{i_mj_m}\}_{\rm LA}$ is also at most $n^{2}-1$.
In other words, we always have
$$\{E_{i_1j_1},\dots,E_{i_mj_m}\}_{\rm LA}\neq\mathfrak{gl}(n).$$
Therefore, if $\{E_{i_1j_1},\dots,E_{i_mj_m}\}_{\rm LA}=\mathfrak{gl}(n)$, then $\mathcal{G}_{\rm contr}^{\ast}$ is a strongly connected digraph with at least one self-loop.

\subsection*{F. Proof of Theorem \ref{thm6}}

\subsubsection*{F.1 Proof of Sufficiency for Theorem \ref{thm6}}
The sufficiency of Theorem \ref{thm6} can be proved in much the same way as Theorem \ref{thm4}, and now we give the main ideas of the proof.

Let $\mathcal{SG}_{\rm contr}$ and $\mathcal{SG}_{\rm drift}$ stand for the simple digraphs corresponding to $\mathcal{G}_{\rm contr}^{\ast}$ and $\mathcal{G}_{\rm drift}^{\ast}$ by ignoring the self-links, respectively.
Because the union graph $\mathcal{G}_{\rm contr}^{\ast}\mcup\mathcal{G}_{\rm drift}^{\ast}$ is strongly connected, the union graph $\mathcal{SG}_{\rm contr}\mcup\mathcal{SG}_{\rm drift}$ is strongly connected too.
Similarly, $\widetilde{A}$ and $\mathcal{SG}_{\rm vaild}$ can be defined. The remainder of the proof is divided into two steps.

\noindent{\it Step 1.}
We first show that the following formula holds.
\begin{equation}\label{eq8}
	\{A, E_{i_1j_1},\dots,E_{i_mj_m}\}_{\rm LA}\supset\mathfrak{sl}(n).
\end{equation}
Because the weakly connected components of $\mathcal{G}_{\rm contr}^{\ast}$ are strongly connected with at least two nodes, and one of them contains at least three nodes, so does $\mathcal{SG}_{\rm contr}$. According to the proof of Theorem \ref{thm4}, we have
\begin{equation*}
	\{A, E_{i_1j_1},\dots,E_{i_mj_m}\}_{\rm LA}\supset\mathfrak{sl}(n).
\end{equation*}

\noindent{\it Step 2.}
In this step, we proceed to show that $\{A, E_{i_1j_1},\dots,E_{i_mj_m}\}_{\rm LA}=\mathfrak{gl}(n)$ holds under the following two cases: (i) $\mathcal{G}_{\rm contr}^{\ast}$ has a self-loop; (ii)  ${\rm tr}A\neq 0$.
\begin{itemize}
	\item Proof under Case (i): if $\mathcal{G}_{\rm contr}^{\ast}$ has a self-loop, then combining (\ref{eq8}) with the proof of Proposition \ref{thm5}, we conclude that
	$$\{A, E_{i_1j_1},\dots,E_{i_mj_m}\}_{\rm LA}=\mathfrak{gl}(n).$$
	
	\item Proof of Case (ii): if ${\rm tr}A\neq 0$, we can certainly assume that $A=\sum_{k=1}^{l}a_{k}E_{i_{k}i_{k}}+\sum_{k=l+1}^{r}a_{k}E_{i_{k}j_{k}}$.
	From (\ref{eq8}) we have
	$$\sum_{k=1}^{l}a_{k}E_{i_{k}i_{k}}\in\{A, E_{i_1j_1},\dots,E_{i_mj_m}\}_{\rm LA}.$$
	Because
	$E_{ii}-E_{jj}\in\{A, E_{i_1j_1},\dots,E_{i_mj_m}\}_{\rm LA}$, for $1\leq i\neq j \leq n$, we can obtain
	\begin{align*}
		\sum_{k=1}^{l}a_{k}E_{i_{k}i_{k}} & +a_{k}(E_{i_{k-1}i_{k-1}}-E_{i_{k}i_{k}}) \\
		& +(a_{k}+a_{k-1})(E_{i_{k-2}i_{k-2}}-E_{i_{k-1}i_{k-1}})\\
		& +\cdots+\sum_{k=2}^{l}a_{k}(E_{i_{1}i_{1}}-E_{i_{2}i_{2}})\\
		& =\sum_{k=1}^{l}a_{k}E_{i_{1}i_{1}}\in \{A, E_{i_1j_1},\dots,E_{i_mj_m}\}_{\rm LA}.
	\end{align*}
	Therefore, when ${\rm tr}A\neq 0$, i.e., $\sum_{k=1}^{l}a_{k}\neq0$, we can obtain
	$$E_{i_{1}i_{1}}\in\{A, E_{i_1j_1},\dots,E_{i_mj_m}\}_{\rm LA}.$$
	Again combining (\ref{eq8}) with the proof of Theorem \ref{thm5}, we conclude that
	$$\{A, E_{i_1j_1},\dots,E_{i_mj_m}\}_{\rm LA}=\mathfrak{gl}(n).$$
\end{itemize}
Consequently, the system \eqref{eq:model3} is accessible on ${\rm GL^{+}}(n)$ by Theorem \ref{basic}.
\subsubsection*{F.2 Proof of Necessity  for Theorem \ref{thm6}}
The proof of necessity for Theorem \ref{thm6} is straightforward along  the lines of the proof of necessity for Theorem \ref{thm4} and Proposition \ref{thm5}.

\subsection*{G. Proof of Theorem \ref{thm7}}
First of all, when $\mathcal{G}_{\rm contr}^{*}$ has a weakly connected component which contains at least three nodes, the theorem holds according to Theorem \ref{thm6}.
Thus, in the rest of the proof, we focus on proving the case when $\mathcal{G}_{\rm contr}^{\ast}$ has no weakly connected component which contains at least three nodes, i.e., each weakly connected component of $\mathcal{G}_{\rm contr}^{\ast}$ contains only two nodes.

An analysis similar to that in the proof of sufficiency for Theorem \ref{thm6} shows that the proof is completed by showing that the formula (\ref{eq8}) still holds when each weakly connected component of $\mathcal{G}_{\rm contr}^{\ast}$ contains only two nodes.
In fact, in this situation, a weakly connected component of $\mathcal{G}_{\rm contr}^{\ast}$ containing a self-loop can replace the role of the weakly connected component of $\mathcal{G}_{\rm contr}^{\ast}$ containing at least three nodes.

Recall that $\mathcal{SG}_{\rm contr}$ and $\mathcal{SG}_{\rm drift}$ represent the simple digraphs corresponding to $\mathcal{G}_{\rm contr}^{\ast}$ and $\mathcal{G}_{\rm drift}^{\ast}$, respectively. The digraph $\mathcal{SG}_{\rm valid}$ satisfies: i) all arcs are between different $\mathrm{V}_{i}$; ii) each $\mathrm{V}_{i}$ has at least one node with out-degree greater than zero; iii) each $\mathrm{V}_{i}$ has at least one node with in-degree greater than zero.
It can be seen from the proof of Theorem \ref{thm4} that we need only consider the case when $\mathcal{G}_{\rm contr}^{\ast}$ has two weakly connected components.

Let $\mathcal{SG}_{\rm contr}=\mathcal{SG}_{\rm contr}^{1}\mcup\mathcal{SG}_{\rm contr}^{2}$ and $\mathrm{V}_{1}=\{v_{11},v_{12}\},\mathrm{V}_{2}=\{v_{21},v_{22}\}$.
Since $\mathcal{G}_{\rm contr}^{\ast}$ has a self-loop, without loss of generality we assume $E_{v_{11}v_{11}}\in\{E_{i_1j_1},\dots,E_{i_mj_m}\}$.
For node $v_{11}\in \mathrm{V}_{1}$ of $\mathcal{SG}_{\rm valid}$, there are three possibility for its degree: (i) $\deg^+(v_{11})=\deg^-(v_{11})=0$; (ii) $\deg^-(v_{11})>0$; (iii) $\deg^+(v_{11})>0$.
\begin{itemize}
	\item Proof under Case (i): if $\deg^+(v_{11})=\deg^-(v_{11})=0$, then $\deg^+(v_{12})>0, \deg^-(v_{12})>0$ and the statement holds obviously.
	\item Proof of Case (ii): if $\deg^-(v_{11})=k>0$, we can use the self-loop to come to the desired conclusion. For any $B\in\mathfrak{gl}(n)$, $B$ can be represented in the form of $\sum_{k=1}^{l} b_{k} E_{{i}_k{j}_k}+\sum_{k=l+1}^{r} b_{k} E_{{i}_k{i}_k}$.
	Define the function $\tilde{\varphi}$ that takes a matrix $B\in\mathfrak{gl}(n)$ to a simple digraph $\mathcal{G}_{B}^{\ast}:=(\mathrm{V},\mathcal{E}_{B}^{\ast})$, where $(i_k,j_k)\in\mathcal{E}_{B}^{\ast}$ if and only if $b_{k}\neq 0$ for $1\leq k\leq l$.
	From Lemma \ref{lem6} we have
	$\tilde{\varphi}([\widetilde{A},E_{v_{11}v_{11}}])$ satisfies: i) $\deg^-(v_{11})=k>0$; ii) $\deg^+(v_{12})=0$.
	From the proof of Theorem \ref{thm4} we have
	$$E_{ij}\in\{A,E_{i_1j_1},\dots,E_{i_mj_m}\}_{\rm LA}$$
	hold for all $i\in\mathrm{V}_{1}, j\in\mathrm{V}_{2}$.
	Moreover, by using these new elements can make $\widetilde{A}$ into $\widetilde{A}_{1}$ through a linear combination such that the graph $\tilde{\varphi}(\widetilde{A}_{1})$ only has arcs from nodes in $\mathrm{V}_{2}$ to nodes in $\mathrm{V}_{1}$. Hence $\tilde{\varphi}(\widetilde{A}_{1})$ has a node in $\mathrm{V}_{2}$ with out-degree greater than zero, and the other node in $\mathrm{V}_{2}$ with zero in-degree. Thus, $$E_{ij}\in\{A,E_{i_1j_1},\dots,E_{i_mj_m}\}_{\rm LA}$$
	holds for all $i\in\mathrm{V}_{2}, j\in\mathrm{V}_{1}$, and finally, we have
	\begin{equation*}
		\{A, E_{i_1j_1},\dots,E_{i_mj_m}\}_{\rm LA}\supset\mathfrak{sl}(n).
	\end{equation*}
	\item Proof of Case (iii): if $\deg^+(v_{11})=k>0$, $\tilde{\varphi}([\widetilde{A},E_{v_{11}v_{11}}])$ satisfies: i) $\deg^+(v_{11})=k>0$; ii) $\deg^-(v_{12})=0$. In the same manner as in Case (ii), we can see that  $\{A, E_{i_1j_1},\dots,E_{i_mj_m}\}_{\rm LA}\supset\mathfrak{sl}(n)$ holds.
\end{itemize}

\end{document}